\documentclass[11pt]{amsart}
\usepackage{amssymb,amsmath,amsthm}
\usepackage{a4wide}
\usepackage{graphicx}

\usepackage{color}
\usepackage{mathrsfs}
\usepackage{latexsym}
\usepackage{bbm}
\usepackage{hyperref}
\usepackage{mathtools}

\numberwithin{equation}{section}

\theoremstyle{plain}
\begingroup
\newtheorem{theorem}{Theorem}[section]
\newtheorem{lemma}[theorem]{Lemma}
\newtheorem{proposition}[theorem]{Proposition}

\endgroup

\theoremstyle{definition}
\begingroup
\newtheorem{definition}[theorem]{Definition}
\newtheorem{remark}[theorem]{Remark}

\endgroup

\newcommand{\F}{\mathcal{F}}

\newcommand{\N}{\mathbb{N}}

\newcommand{\R}{\mathbb{R}}

\newcommand{\ffi}{\varphi}

\newcommand{\ud}{\,\mathrm{d}}

\newcommand{\weakly}{\rightharpoonup}

\newcommand{\ep}{\varepsilon}

\newcommand{\cc}{\mathrm{c}}
\newcommand{\CC}{\mathrm{C}}
\newcommand{\B}{\mathcal B}
\newcommand{\h}{\mathscr{H}}

\definecolor{ddorange}{rgb}{1,0.5,0}
\definecolor{ddcyan}{rgb}{0,0.2,1.0}

\title[The variational approach to $s$-fractional heat flows] {The variational approach to $s$-fractional heat flows and the limit cases $s\to 0^+$ and $s\to 1^-$}

\author[V. Crismale]
{V. Crismale}
\address[Vito Crismale]{Dipartimento di Matematica ``Guido Castelnuovo'', Sapienza Universit\`a di Roma, Piazzale Aldo Moro 2, I-00185 Roma, Italy
}
\email[V. Crismale]{crismale@mat.uniroma1.it}

\author[L. De Luca]
{L. De Luca}
\address[Lucia De Luca]{IAC-CNR, Via dei Taurini 19, I-00185 Roma, Italy}
\email[L. De Luca]{lucia.deluca@cnr.it}

\author[A. Kubin]
{A. Kubin}
\address[Andrea Kubin]{Dipartimento di Matematica ``Guido Castelnuovo'', Sapienza Universit\`a di Roma, Piazzale Aldo Moro 2, I-00185 Roma, Italy
}
\email[A. Kubin]{kubin@mat.uniroma1.it}

\author[A. Ninno]
{A. Ninno}
\address[Angelo Ninno]{Dipartimento di Matematica ``Guido Castelnuovo'', Sapienza Universit\`a di Roma, Piazzale Aldo Moro 2, I-00185 Roma, Italy
}
\email[A. Ninno]{angelo.ninno@uniroma1.it}

\author[M. Ponsiglione]
{M. Ponsiglione}
\address[Marcello Ponsiglione]{Dipartimento di Matematica ``Guido Castelnuovo'', Sapienza Universit\`a di Roma, Piazzale Aldo Moro 2, I-00185 Roma, Italy
}
\email[M. Ponsiglione]{ponsigli@mat.uniroma1.it}
\begin{document}
 \maketitle

\begin{abstract}
This paper deals with the limit cases for $s$-fractional heat flows in a cylindrical domain, with  homogeneous Dirichlet boundary conditions, as $s\to 0^+$ and $s\to 1^-$\,. 

To this purpose, we describe the fractional heat flows as minimizing movements of the corresponding Gagliardo seminorms, with respect to the $L^2$ metric. First, we provide an abstract stability result for minimizing movements  in Hilbert spaces, with respect to a sequence of $\Gamma$-converging uniformly $\lambda$-convex  energy functionals. Then, we provide the $\Gamma$-convergence analysis of the $s$-Gagliardo seminorms as $s\to 0^+$ and $s\to 1^-$\,, and apply the general stability result to such specific cases. 

As a consequence, we prove that $s$-fractional heat flows (suitably scaled in time) converge to the standard heat flow as $s\to 1^-$, and to a degenerate ODE type flow as $s\to 0^+$\,. Moreover, looking at the next order term in the asymptotic expansion of the $s$-fractional Gagliardo seminorm,  we show that suitably forced $s$-fractional heat flows converge, as $s\to 0^+$\,,  to the parabolic flow of an energy functional that can be seen as a sort of renormalized $0$-Gagliardo seminorm: the resulting parabolic equation involves the first variation of such an energy, that can be understood as a zero (or logarithmic) Laplacian.    

\end{abstract}
\setcounter{tocdepth}{1} 
\tableofcontents
\section*{Introduction}
 In this paper we consider the $s$-fractional heat equation ($s\in (0,1)$)
\begin{equation}\label{equazero}
u_t(t) + C(s) (-\Delta)^s u(t)=0\,,\qquad t\ge 0
\end{equation}
posed in a bounded set $\Omega\subset \R^d$ with homogeneous Dirichlet conditions, and we study its asymptotic behavior as $s\to 0^+$ and $s\to 1^-$.
We adopt a purely variational approach, combining the theory of \emph{minimizing movements} for \emph{gradient flows} and $\Gamma$-convergence properties for the underlying energies.

The fractional heat equation may be seen as the $L^2$-gradient flow of the squared $s$-Gagliardo seminorm
$$
[u]^2_{s}:=\int_{\R^d}\int_{\R^d}\frac{|u(x)-u(y)|^2}{|x-y|^{d+2s}}\ud y\ud x\,,
$$
with the support of $u$ contained in $\Omega$ when the equation is posed in a bounded domain with homogeneous Dirichlet boundary conditions.

The asymptotic behavior of $s$-Gagliardo seminorms has been studied by several authors.
 The case $s\to 1^-$ has been first considered in \cite{BBM1},  where it is proven that the pointwise limit of the squared $s$-Gagliardo seminorms multiplied by $(1-s)$ is given by (a multiple of) the Dirichlet integral.  Such  a result  is indeed  proven for every exponent $1<p<+\infty$  ($[\cdot]_s$ corresponds to $p=2$). For $p=1$ only a control of the limit in terms of the total variation is provided, allowing to characterize the BV space; this  has been extended in several directions: first, in \cite{D}, showing that the pointwise limit is exactly (a multiple of) the total variation, and then in \cite{P, LS} for more general kernels. The asymptotic behavior of (relative)
fractional perimeters (roughly said, considering characteristic functions in the framework of \cite{D}) is provided  in \cite{ADPM} in terms of $\Gamma$-convergence.

For what concerns the limit as $s\to 0^+$\,, the literature is much poorer. In \cite{MS} the authors show that, as $s\to 0^+$\,, the squared $s$-fractional Gagliardo seminorms multiplied by $s$ pointwise converge to (a multiple of) the squared $L^2$-norm
(see also \cite{DFPV} for a similar result in the context of $s$-fractional perimeters). 
The corresponding asymptotic analysis in terms of $\Gamma$-convergence has been developed in \cite{DNP} in the context of fractional perimeters (that is, restricting to characteristic functions). A functional with more interesting properties is obtained in the limit as $s\to 0^+$ by studying the next order term in the asymptotic expansion of the squared $s$-fractional Gagliardo seminorms: in \cite{DNP} it is shown, still restricting to fractional perimeters, that the corresponding $\Gamma$-limit provides a new nonlocal energy, referred to as {\it $0$-fractional perimeter}.

The first step in our analysis is to extend the results in \cite{DNP} to the seminorms. In fact, we remove the constraint on the admissible functions to be characteristic functions
and we obtain a $\Gamma$-convergence expansion (as $s\to 0^+$) with respect to the  $L^2$ topology for functions whose support is in a bounded set $\Omega$\,. 
The zero-order $\Gamma$-convergence result is Theorem~\ref{gammafs}, while 
the next order analysis is done in  Theorem~\ref{thm:stozerofo}.
The $\Gamma$-convergence analysis  
of the (scaled) $s$-Gagliardo seminorms as $s\to 1^-$ is done in Theorem~\ref{thmGammaconvs1}, giving back the $\Gamma$-convergence version of the result in \cite{BBM1}.

The above asymptotic results, which are of independent interest,
 serve as a tool for the stability of the corresponding parabolic flows.
Stability of gradient flows with respect to the $\Gamma$-convergence of the corresponding energies is nowadays a classical problem, which has been widely investigated in recent years in increasing generality (we refer, for instance, to \cite{SS, RS, AGS}). 
In the present paper we exploit the fact that the underlying energies are  uniformly (with respect to $s$)  $\lambda$-convex.

The gradient flows of $\lambda$-convex energies, namely energies which are convex up to a quadratic perturbation multiplied by $\lambda$, admits a unique solution. 
Moreover, such a solution can be approximated by the discrete-in-time implicit Euler scheme; namely, the solution coincides with the so called  minimizing movements solution.
This well-known fact is the content of Theorem~\ref{exist}, which provides also quantitative convergence estimates for   the discrete-in-time solutions to the unique solution of the parabolic flow, the convergence rates depending  essentially on $\lambda$ and on the initial data  (see also \cite{A, DS, NSV}). 
The uniformity of the convergence rates is the key tool for proving Theorem \ref{genstab}, which guarantees stability  for gradient flows in Hilbert spaces with respect to  sequences of $\Gamma$-converging uniformly $\lambda$-convex functionals.

With Theorems \ref{exist} and \ref{genstab} on hand, we are in a position to provide our stability results for the $s$-fractional heat flows as $s\to 0^+$ and $s\to 1^-$\,.

First, we apply Theorem~\ref{exist} in order to show that for every $s\in (0,1)$ there exists a unique solution to the Cauchy problem governed by the $s$-fractional parabolic flow \eqref{equazero} and that such a solution is the limit, as the time step vanishes, of the discrete-in-time evolution.
Actually, the parabolic flow in the abstract result Theorem \ref{exist} is written - as usual - in terms of a differential inclusion of $-u_t$ in the subdifferential of the underlying energy evaluated in $u$, which in general could be multivalued. This is not our case,  since the $s$-Gagliardo seminorms are differentiable in the fractional Sobolev spaces $\mathcal{H}^s_0(\Omega)$\,, which are dense in $L^2(\Omega)$\,; this implies (see  Proposition~\ref{singleton}) that along the parabolic flow, the subdifferential of the nonlocal energy reduces to a point for a.e. $t>0$\,. Actually existence, uniqueness and regularity for the fractional heat flows have been extensively studied; we refer the interested reader to \cite{CCV, FR} and the references therein.
Moreover, in the framework of nonnegative solutions for the fractional heat equation in $\R^d$, the problem has been studied in the context of a general Widder theory \cite{Wid44}, in e.g.\ \cite{BarPerSorVal13, BonSirVaz16, Vaz17}, with even not regular (but nonnegative) initial datum.

Finally, since the energies we consider are uniformly $\lambda$-convex, we can apply the abstract stability result Theorem \ref{genstab} to obtain the asymptotic behavior of the $s$-fractional heat flows as $s\to 0^+$ and $s\to 1^-$\,. 
 This is contained in our main results, Theorems~\ref{convheat0ordsto0}, \ref{convheat1ordsto0}, \ref{11genstab}: the limit evolutions are an ODE corresponding to an exponential growth for the $0$-th order as $s\to 0^+$, a $0$-fractional heat equation for the first order as $s\to 0^+$, and the classical local heat equation as $s\to 1^-$. The stability consists in a weak  $H^1$-in-time-convergence, which is proven to be strong if the initial data are well prepared, namely if the approximating initial data are a recovery sequence for the limit datum with respect to the $\Gamma$-converging energies. Furthermore, in this case for every time $t$ the approximating evolutions are recovery sequences (with respect to the $\Gamma$-converging energies) for the limit evolution,
 namely there is convergence of the corresponding energies for every $t$.

A result similar in spirit to our asymptotic analysis for $s\to 1^-$ is provided in \cite{AliAnsBraPiaTri20}, where the authors consider limit of convolution type energies and of the corresponding parabolic flows as the support of the convolution kernel concentrates to a point. In that paper the authors focus on homogeneous Neumann boundary conditions, and are able to deal also with anisotropic kernels; we believe that, up to minor differences in the setting, their approach and the one developed in this paper are consistent and that our analysis as $s\to 1^-$ could be derived by their analysis with minor modifications. As well, we believe that our results (also for $s\to 0^+$) could be generalized to natural variants of the fractional heat flows considered here, dealing for instance with different boundary conditions as well as anisotropic variants of the underlying Gagliardo seminorms. 

The paper is structured as follows: in Sections~\ref{Sec1} and \ref{Sec2} we study the $\Gamma$-convergence of the functionals as $s\to 0^+$ and $s\to 1^-$\,, respectively. Section~\ref{sec:general} collects the abstract stability result for gradient flows of uniformly $\lambda$-convex $\Gamma$-converging energies. In Section~\ref{Sec4} we prove the main results of convergence for $s$-fractional parabolic flows as $s\to 0^+$ and $s\to 1^-$. 
\vskip20pt
 
\textsc{Acknoledgments:}  The authors are members of the Gruppo Nazionale per l'Analisi Matematica, la Probabilit\`a e le loro Applicazioni (GNAMPA) of the Istituto Nazionale di Alta Matematica (INdAM).
 \vskip20pt

\section{$\Gamma$-convergence of the squared $s$-Gagliardo seminorms as $s\to 0^+$}\label{Sec1}
In this section we analyze the asymptotic behaviour of the scaled and renormalized $s$-Gagliardo seminorms as $s\to 0^+$.

Let  $d\in\N$, $d\geq 1$,  
and let $s\in(0,1)$\,. The Gagliardo $s$-seminorm  of a measurable function $u:\R^d\to\R$ is defined by
$$
[u]_{s}:=\bigg[\int_{\R^d}\int_{\R^d}\frac{|u(x)-u(y)|^2}{|x-y|^{d+2s}}\ud y\ud x\bigg]^{\frac{1}{2}}\,,
$$
whenever the double integral above is finite.
Let $\Omega$ be an open bounded subset of $\R^d$ with Lipschitz continuous boundary.
We denote by $\mathcal{H}^s_0(\Omega)$ the completion of $C^\infty_{\mathrm{c}}(\Omega)$ with respect to the Gagliardo $s$-seminorm defined above.
For every measurable function $u:\Omega\to\R$ we denote by $\tilde u$ its extension to $0$ on the whole $\R^d$\,, i.e., defined by $\tilde u=u$ in $\Omega$ and $\tilde u=0$ in $\R^d\setminus\Omega$\,.

In \cite[Theorem 2]{MS} it has been proven that there exists a constant $C(d)$ depending only on the dimension $d$, such that for  $d > 2s$   
\begin{equation*}
\int_{\Omega}\frac{|u(x)|^2}{|x|^{2s}}\ud x\le C(d)\frac{s(1-s)}{(d-2s)^2}[\tilde u]^2_{s} \,\quad\textrm{ for  every }u\in \mathcal{H}^s_0(\Omega)\,.
\end{equation*}
It follows that $\mathcal{H}^s_0(\Omega)\subset L^2(\Omega)$  for every $d\ge 1$  and every $s\in(0,1)$:  for $2s <d$ this comes from the above estimate, being $\Omega$ bounded; for $d\leq 2s$ it is enough to pass to suitable $s'< s$ with $2s' <d$, recalling that $[\tilde{u}]_{s_1}\leq C(d,s)[\tilde{u}]_{s_2}$
for $0<s_1\leq s_2<1$ 
(see e.g.\ \cite[Proposition~2.1]{DPV}). 

Along with \cite[Theorem~1.4.2.2]{Gris} (see also \cite[Theorem~3.29]{McLean}), 
the inclusion $\mathcal{H}^s_0(\Omega)\subset L^2(\Omega)$ gives that 
\begin{equation*}
\mathcal{H}^s_0(\Omega)=\left\{u\in L^2(\Omega)\,:\,[\tilde u]_s<+\infty\right\}\,.
\end{equation*}
For every $s\in(0,1)$, we define the functional $F^{s}:L^2(\Omega)\to [0,+\infty]$ as 
\begin{equation}\label{defFs}
F^{s}(u):=[\tilde u]^2_s\,.
\end{equation}
\subsection{$0$-th order $\Gamma$-convergence for the functionals $F^{s}$ as $s\to 0^+$}
We define the functional $F^{0}:L^2(\Omega)\to [0,+\infty)$ as
\begin{equation}\label{F0}
F^{0}(u):=\frac{d\omega_d}{2}\|u\|^2_{L^2},
\end{equation}
where $\omega_d$ is the measure of the unit ball of $\R^d$.
The following result is a trivial consequence of \cite[formula (9)]{MS}.
\begin{theorem}\label{thm:mazsha}
Let $\delta\in(0,1)$\,. For every $s\in\big(0,\frac{\delta^2}{8}\big)$ and for every $u\in \mathcal{H}^s_0(\Omega)$ we have 
\begin{equation}\label{eq:mazsha}
\frac{d\omega_d}{2}\int_{\Omega}\frac{|u(x)|^2}{|x|^{2s}}\ud x\le s \frac{2^{2s}}{(1-\delta)^2}F^{s}(u)\,.
\end{equation}
\end{theorem}
The following theorem follows easily from the above estimate.
\begin{theorem}\label{gammafs}
Let $\{s_n\}_{n\in\N}\subset(0,1)$ be such that $s_n\to 0^+$ as $n\to +\infty$. 
\begin{itemize}
\item[(i)](Compactness) Let $\{u^n\}_{n\in\N}\subset L^2(\Omega)$ be such that 
$$
\sup_{n\in\N}s_nF^{s_n}(u^n)\le C,
$$ 
for some constant $C\in\R$. Then, up to a subsequence, $u^n\weakly u$ in $L^2(\Omega)$ for some $u\in L^2(\Omega)$.
\item[(ii)]($\Gamma$-liminf inequality) For every $u\in L^2(\Omega)$ and for every $\{u^n\}_{n\in\N}\subset L^2(\Omega)$ with $u^n\weakly u$ in $L^2(\Omega)$, it holds
\begin{equation*}
F^{0}(u)\le \liminf_{n\to +\infty}s_nF^{s_n}(u^n).
\end{equation*}
\item[(iii)]($\Gamma$-limsup inequality) For every $u\in L^2(\Omega)$ there exists a sequence $\{u^n\}_{n\in\N}\subset L^2(\Omega)$ with 
$u^n\to u$ in $L^2(\Omega)$ such that
\begin{equation*}
F^{0}(u)= \lim_{n\to +\infty}s_nF^{s_n}(u^n).
\end{equation*}
\end{itemize}
\end{theorem}
\begin{proof}
Since $\Omega$ is bounded, there exists $0<R<+\infty$ such that $\Omega\subset B_R$\,.
Therefore, in view of \eqref{eq:mazsha} and of the energy bound, for $n$ large enough, we have that
\begin{equation}\label{triv1}
\begin{aligned}
 \frac{1}{R^{2s_n}} \int_{\Omega}|u^n(x)|^2\ud x \le \int_{\Omega}\frac{|u^n(x)|^2}{|x|^{2s_n}}\ud x 
 \le C(d)\,.
\end{aligned}
\end{equation}
It follows that $\|u^n\|_{L^2(\Omega)}$ is uniformly bounded and hence, up to a subsequence, $u^n\weakly u$ in $L^2(\Omega)$ for some $u\in L^2(\Omega)$\,, proving (i). 
\vskip4pt
Let us pass to the proof of (ii).
Let $\delta\in(0,1)$ be fixed. Using again 
\eqref{eq:mazsha}, for $n$ large enough, we have
 \begin{equation*}
\begin{aligned}
s_n F^{s_n}(u^n)\ge  \frac{(1-\delta)^2}{2^{2s_n}}\frac{d\omega_d}{2}\int_{\Omega}\frac{|u^n(x)|^2}{|x|^{2s_n}}\ud x\ge \frac{ (1-\delta)^2}{2^{2s_n}R^{2s_n }}\frac{d\omega_d}{2}\int_{\Omega}|u^n(x)|^2\ud x,
\end{aligned}
\end{equation*}
which, passing to the limit as $n\to +\infty$ and using the weak lower semicontinuity of the $L^2$ norm, yields
$$
\liminf_{n\to +\infty} s_n F^{s_n}(u^n)\ge (1-\delta)^2 \frac{d\omega_d}{2}\int_{\Omega}|u(x)|^2\ud x;
$$
by the arbitrariness of $\delta$, the claim (ii) follows.
\vskip4pt
Now we show that also (iii) holds true. If $u\in \CC^{\infty}_{\cc}(\Omega)$\,, the claim is proven in \cite[Theorem 3]{MS}, with $u^n\equiv u$\,. Since  $ \CC^{\infty}_{\cc}(\Omega)$ is dense in $L^2(\Omega)$\,, the general case follows by a standard diagonal argument. 
\end{proof}
\subsection{The first order $\Gamma$-limit of the functionals $F^{s}$ as $s\to 0^+$}
In order to compute the $\Gamma$-limit of the renormalized functionals
$F^{s}-\frac{1}{s}F^{0}$ as $s\to 0^+$ we need to rewrite the functional $F^{s}$ in a different manner.
Let $s\in [0,1)$.
We define the functional $G^{s}_{1}:L^2(\Omega)\to [0,+\infty]$ as
\begin{equation}\label{HsR}
\displaystyle G^{s}_1(u):=\frac 1 2\iint_{\B_1}\frac{|\tilde u(x)-\tilde u(y)|^2}{|x-y|^{d+2s}}\ud y\ud x\,,
\end{equation}
where $\B_1:=\{(x,y)\in\R^{d}\times\R^d\,:\,|x-y|<1\}$\,,
and the functional  $J^s_1:L^2(\Omega)\to (-\infty,+\infty)$ as
\begin{equation}\label{Jsr}
J^s_1(u):=-\iint_{\R^{2d}\setminus \overline\B_1}\frac{\tilde u(x)\tilde u(y)}{|x-y|^{d+2s}}\ud y\ud x\,.
\end{equation}
We notice that the functionals $J^s_1$ are well-defined in $L^2(\Omega)$ since, by H\"older inequality,
\begin{equation}\label{estJsr}
|J^s_{1}(u)|\le \|u\|_{L^1(\Omega)}^2\le |\Omega|\|u\|^2_{L^2(\Omega)}\,.
\end{equation}
It is easy to check that for every $s\in (0,1)$
\begin{equation}\label{reprs}
\hat F^s(u):=F^s(u)-\frac 1 s F^0(u)=G^s_1(u)+J^s_1(u)\qquad\textrm{for every }u\in L^2(\Omega).
\end{equation}
In analogy with \eqref{reprs}, we define the functionals $\hat F^0:L^2(\Omega)\to (-\infty,+\infty]$ as
\begin{equation}\label{repr0}
\hat F^0(u):=G^0_1(u)+J^0_1(u)\,,
\end{equation}
and we introduce the space
$$
\mathcal{H}_0^{0}(\Omega):=\{u\in L^2(\Omega)\,:\,G^0_1(u)<+\infty\}\,.
$$
\begin{remark}
\rm{It is natural to endow the space $\mathcal{H}_0^0(\Omega)$ with a {\it 0-Gagliardo} type norm
$$
[u]_0:=(2 G^0_1(u))^{\frac{1}{2}}\,.
$$
}
\end{remark}
We are now in a position to state our $\Gamma$-convergence result for the functionals
$\hat F^s$ defined in \eqref{reprs}\,.
\begin{theorem}\label{thm:stozerofo}
Let $\{s_n\}_{n\in\N}\subset(0,1)$ be such that $s_n\to 0^+$ as $n\to +\infty$\,. The following $\Gamma$-convergence result holds true.
\begin{itemize}
\item[(i)] (Compactness) Let $\{u^n\}_{n\in\N}\subset L^2(\Omega)$ be such that
\begin{equation}\label{COMP2}
\hat F^{s_n}(u^n)+ 2|\Omega| \| u^n \|^2_{L^2(\Omega)}\le M,
\end{equation} 
for some constant $M$ independent of $n$\,. Then, up to a subsequence, $u^n\to u$ strongly in $L^2(\Omega)$ for some $u\in \mathcal{H}_0^0(\Omega)$\,.
\item[(ii)] ($\Gamma$-liminf inequality) For every $u\in L^2(\Omega)$ and for every $\{u^n\}_{n\in\N}\subset L^2(\Omega)$ with $u^n\to u$ in $L^2(\Omega)$\,, it holds
$$
\hat F^0(u)\le\liminf_{n\to +\infty}\hat F^{s_n}(u^n)\,.
$$
\item[(iii)] ($\Gamma$-limsup inequality) For every $u\in\mathcal{H}_0^0(\Omega)$ there exists $\{u^n\}_{n\in\N}\subset L^2(\Omega)$ with $u^n\to u$ in $L^2(\Omega)$ such that
$$
\hat F^0(u)=\lim_{n\to +\infty}\hat F^{s_n}(u^n)\,.
$$
\end{itemize}
\end{theorem}
\begin{remark}\label{DCT}
	Notice that for all $s\in [0,1)$ and for all $u,\, v \in L^2(\Omega)$ we have
\begin{equation}\label{DCT2}
\begin{aligned}
|J^s_1(u) - J^s_1(v)| \le &  (\|u\|_{L^1(\Omega)} + \|v\|_{L^1(\Omega)})\|u-v\|_{L^1(\Omega)}
\\
\le & |\Omega| (\|u\|_{L^2(\Omega)} + \|v\|_{L^2(\Omega)})\|u-v\|_{L^2(\Omega)}\, . 
\end{aligned}
\end{equation}
	
In particular, for all $s\in [0,1)$ the functionals $J^s_1$ are continuous with respect to the strong  $L^2$ convergence.
\end{remark}
\subsection{Compactness and $\Gamma$-liminf inequality}
In order to prove (i) of Theorem \ref{thm:stozerofo}, we recall the following result proven in \cite{JW} .
\begin{theorem}[Local compactness \cite{JW}]\label{jw}
Let $k:\R^d\to [0,+\infty]$ be a radially symmetric kernel such that
$$
\int_{\R^d}k(z)\ud z=+\infty\qquad\textrm{and}\qquad\int_{\R^d}\min\{1,|z|^2\}k(z)\ud z<+\infty
$$ 
and let
$$
\mathcal{W}^{k}(\Omega):=\left\{u\in L^2(\Omega)\,:\,\iint_{\R^{2d}}|\tilde u(x)-\tilde u(y)|^2k(x-y)\ud y\ud x<+\infty\right\}
$$
be the Banach space endowed with the norm
$$
\|u\|_{\mathcal{W}^k(\Omega)}:=\|u\|_{L^2(\Omega)}+\Big(\iint_{\R^{2d}}|\tilde u(x)-\tilde u(y)|^2k(x-y)\ud y\ud x\Big)^{\frac 1 2}\,.
$$
Then, the embedding $\mathcal{W}^{k}(\Omega)\hookrightarrow L^2(\Omega)$ is compact.
\end{theorem}
With Theorem \ref{jw} on hand, we are in a position to prove compactness.
\begin{proof}[Proof of Theorem \ref{thm:stozerofo}(i)]
By \eqref{COMP2}, \eqref{reprs} and \eqref{estJsr}, we have that
$$
M\ge \hat F^{s_n}(u^n)+2|\Omega|\|u^n\|^2_{L^2(\Omega)}\ge -|\Omega|\|u^n\|_{L^2(\Omega)}^2+ 2|\Omega|\|u^n\|^2_{L^2(\Omega)}= |\Omega|\|u^n\|^2_{L^2(\Omega)}\,,
$$
i.e., that $\|u^n\|_{L^2(\Omega)}$ is uniformly bounded.
Therefore, by \eqref{reprs} we deduce
	\begin{equation}\label{g01}
	\begin{aligned}
	G^0_1(u^n)\le& G^{s_n}_1(u^n)\le M+\iint_{\R^{2d}\setminus\overline\B_1}\frac{|\tilde u^n(x)||\tilde u^n(y)|}{|x-y|^{d+2s_n}}\ud y\ud x
	\le M+|\Omega|\|u^n\|^2_{L^2(\Omega)}\\
	\le& 2 M\,,
	\end{aligned}
	\end{equation}
	whence, by applying Theorem \ref{jw} with $k(z):=\frac{\chi_{B_1}(z)}{|z|^d}$\,, we deduce that, up to a subsequence, $u^n\to u$ in $L^2(\Omega)$ for some $u\in L^2(\Omega)$\,.
Finally, by \eqref{g01} and by the lower semicontinuity of the functional $G^0_1$ with respect to the strong $L^2$ convergence, we get that $u\in\mathcal{H}_0^0(\Omega)$\,.
\end{proof}
Now we prove the $\Gamma$-liminf inequality.
\begin{proof}[Proof of Theorem \ref{thm:stozerofo}(ii)]
By Fatou lemma we have 
\begin{equation*} 
G^0_1(u) \le \liminf_{n\to +\infty}G^{s_n}_1(u^n)\,;
\end{equation*} 
moreover, by
\eqref{DCT2} we easily get
\begin{equation*} 
J^0_1(u)=\lim_{n\to +\infty}J^{s_n}_1(u^n)\,.
\end{equation*} 
In view of \eqref{reprs} and \eqref{repr0}, we get the claim.
\end{proof}
\subsection{$\Gamma$-limsup inequality}
Here we construct the recovery sequence for the functionals $\hat F^s$\,.
 We start by showing that, for smooth functions, the pointwise limit of the functional $\hat F^s$ as $s\to 0^+$ coincides with the functionals $\hat F^0$\,.
\begin{lemma}\label{ficc}
For every $u\in \CC^{\infty}_{\cc}(\Omega)$ we have that
$$
\lim_{s\to 0^+} \hat F^s(u)=\hat F^0(u)\,.
$$
\end{lemma}
\begin{proof}
In view of the definition of $\hat F^s$ in \eqref{reprs} it is enough to show
\begin{eqnarray}\label{qls1}
\lim_{s\to 0^+}G^s_1(u)=G^0_1(u)\,,\\ \label{qls2}
\lim_{s\to 0^+}J^s_1(u)=J^0_1(u)\,.
\end{eqnarray}
We start by proving \eqref{qls1}.
To this end, we note that, since $\tilde u\in \CC^\infty(\R^d)$, for every $x, \, y\in \R^d$ we have 
\begin{equation*}
|\tilde u(y)-\tilde u(x)|^2\le \|\nabla \tilde u\|_{L^\infty}^2 |x-y|^2 \, .
\end{equation*}

Let $U\subset \R^d$ be an open set such that $\mathrm{dist}(\Omega,\R^d\setminus U)>1$ and let  $\ep\in(0,1)$; we have

\begin{equation*}
\begin{aligned}
|G^s_1(u)-G^0_1(u)|\le  
&\frac 1 2 \int_{U}\ud x\int_{B_\ep(x)}|\tilde u(x)-\tilde u(y)|^2\bigg|\frac{1}{|x-y|^{d+2s}}-\frac{1}{|x-y|^d}\bigg|\ud y\\
&+\frac{1}{2}\iint_{\B_1\setminus \overline\B_\ep}|\tilde u(x)-\tilde u(y)|^2\bigg|\frac{1}{|x-y|^{d+2s}}-\frac{1}{|x-y|^d}\bigg|\ud y\ud x.
\end{aligned}
\end{equation*}
By Dominated Convergence Theorem the second addend in the righthand side tends to zero (for fixed $\ep$) as $s\to 0^+$, while the first addend is bounded from above by 
$|U|\|\nabla\tilde u\|^2_{L^\infty}\int_{B_\ep}\frac{1}{|z|^{d+2s-2}}\ud z$, which tends to zero as $\ep\to 0^+$. This clearly yields \eqref{qls1}.

Finally, \eqref{qls2} is a trivial consequence of the Dominated Convergence Theorem, once noticed that
$$
J^s_1(u)=-\int_{\Omega}u(x)\int_{\Omega\setminus \overline{B}_1(x)}\frac{u(y)}{|x-y|^{d+2s}}\ud y\ud x\,.
$$
\end{proof}
\begin{lemma}[Density of smooth functions]\label{lm:density}
For every $u\in \mathcal{H}_0^0(\Omega)$ there exists $\{u_k\}_{k\in\N}\subset \CC^\infty_{\cc}(\Omega)$ such that $u_k\to u$ (strongly) in $L^2$
and
\begin{equation*}
\lim_{k\to +\infty}J^0_1(u_k)=J^0_1(u)\qquad\textrm{and}\qquad\lim_{k\to +\infty} G^0_1(u_k)= G_1^0(u).
\end{equation*}
\end{lemma}
\begin{proof}
 This result is proven in \cite[Theorem~3.29]{McLean}, for domains with a continuous boundary. For the reader's convenience we recall the sketch of the proof: Up to a partition of the unity argument, one may assume $\Omega$ to be the subgraph of a continuous function:
thus it is enough to approximate first with $u_\delta(x):=u(x', x_n+\delta)$, for small $\delta$, whose support is well contained in $\Omega$, and then to take $u_\delta \ast \phi_\varepsilon$, for a family of mollifiers $\{\phi_\varepsilon\}_\varepsilon$ and small $\varepsilon$. 
\end{proof}
The limsup inequality in Theorem~\ref{thm:stozerofo} follows directly from the density proved above.
\begin{proof}[Proof of Theorem \ref{thm:stozerofo}(iii)]
Let $u\in \mathcal{H}^0(\Omega)$\,. By Lemma \ref{lm:density} there exists a sequence of functions  $\{u^k\}_{k\in\N}\subset\CC^\infty_\cc(\R^d)$ such that $u^k\to u$ in $L^2$ and
$$
\limsup_{k\to +\infty}\hat F^0(u^k)= \hat F^0(u)\,.
$$
In view of Lemma \ref{ficc} we have 
$$
\lim_{n\to +\infty}\hat F^{s_n}(u^k)=\hat F^{0}(u^k)\qquad\textrm{for every }k\in\N\,.
$$
Therefore, by a standard diagonal argument, there exists a sequence $\{u^n\}_{n\in\N}\subset \CC^{\infty}_\cc(\R^d)$ with $u^n=u^{k(n)}$ for every $n\in\N$ satisfying the desired properties.  
\end{proof}


\section{$\Gamma$-convergence of the squared $s$-Gagliardo seminorms as $s \rightarrow 1^-$}\label{Sec2}
Here we study the $\Gamma$-convergence of the functionals $(1-s)F^s$ as $s \rightarrow 1^-$, where $F^s$ is defined in \eqref{defFs}. 
The candidate $\Gamma$-limit is the functional $F^1: L^2(\Omega) \rightarrow \mathbb{R} \cup \{+ \infty\}$ defined by
\begin{equation*}
F^1(u):=
\left\{
\begin{aligned}
&\frac{\omega_{d}}{4} \int_{\Omega} \vert \nabla u(x) \vert^2 \ud x& \text{ if }  u \in H_0^1(\Omega) \, , \\
& +\infty & \textrm{elsewhere in }L^2(\Omega)\,.
\end{aligned}
\right.
\end{equation*}
\begin{theorem}\label{thmGammaconvs1}
	Let $\{s_n\}_{n \in \mathbb{N}} \subset (0,1)$ be such that $s_n \rightarrow 1^-$ as $ n \to + \infty$. The following $\Gamma$-convergence result holds true.
	\begin{itemize}
		\item[(i)] (Compactness) Let $\{u^n\}_{n\in\N}\subset L^2(\Omega)$ be such that
		\begin{equation}\label{COMP3}
		\sup_{n \in \mathbb{N}} (1-s_n)F^{s_n}(u^n)+ \| u^n \|_{L^2(\Omega)}^2\le M,
		\end{equation} 
		for some constant $M$ independent of $n$\,. Then, up to a subsequence, $u^n\to u$ strongly in $L^2(\Omega)$ for some $u\in H_0^1(\Omega)$\,.
		\item[(ii)] ($\Gamma$-liminf inequality) For every $u\in L^2(\Omega)$ and for every $\{u^n\}_{n\in\N}\subset L^2(\Omega)$ with $u^n\to u$ in $L^2(\Omega)$\,, it holds
		\begin{equation*}
		 F^1(u)\le\liminf_{n\to +\infty}(1-s_n) F^{s_n}(u^n)\,.
		\end{equation*}
		\item[(iii)] ($\Gamma$-limsup inequality) For every $u\in L^2(\Omega)$ there exists $\{u^n\}_{n\in\N}\subset L^2(\Omega)$ with $u^n\to u$ in $L^2(\Omega)$ such that
		\begin{equation}\label{lsupsto1}
		 F^1(u)=\lim_{n\to +\infty} (1-s_n)F^{s_n}(u^n)\,.
		\end{equation}
	\end{itemize}
	\end{theorem}
\subsection{Proof of Compactness}
This subsection is devoted to the proof of Theorem \ref{thmGammaconvs1}(i). 
To accomplish this task, we adopt the strategy in \cite{ADPM} adapting it to our case.
To this purpose, for every function $v\in L^2(\Omega)$ and for every $h\in\R^d$  we denote by $\tau_hv$ the shift of $v$ by $h$,   defined by $\tau_h v (\cdot):= v(\cdot+h)$\,. 
We recall the following two classical results.
\begin{theorem}[Fr\'echet-Kolmogorov]\label{frechetkolmthm}
	Let $\{v^n\}_{n \in \mathbb{N}}\subset L^2(\R^d)$ be such tha $\sup_{n\in\N}\|v^n\|_{L^2(\R^d)}\le M$\,, for some constant $M$ independent of $n$\,.
	If
	$$
	 \lim_{\vert h \vert \rightarrow 0^+} \sup_{n \in \mathbb{N}} \| \tau_{h}v^n-v^n \|_{L^2(\mathbb{R}^d)}=0\,,
	 $$
	then $ \{v^n\}_{n \in \mathbb{N}}$ is pre-compact in $L_{\mathrm{loc}}^2(\mathbb{R}^d)$\,. 
\end{theorem}
\begin{theorem}\label{carsobvshift}
	Let $v \in L^2(\mathbb{R}^d)$. Then $v \in H^1(\mathbb{R}^d)$ if and only if there exists $C>0$ such that 
	$$ 
\| \tau_{h}v-v \|_{L^2(\R^d)} \leq C \vert h \vert\quad \textrm{  for every }h\in\R^d\,.
$$

\end{theorem}
For every $A \subset \mathbb{R}^d$ and for every $t >0$ we define the set
\begin{equation}\label{ingrass}
A_{t}:= \{ x \in \mathbb{R}^d: \; \mathrm{dist}(x,A)<t\}.
\end{equation}
The following result which allows to estimate the $L^2$ distance of a function from its shift has been proven in \cite[Proposition 5]{ADPM} in $L^1$\,; for the sake of completeness,   we state
and prove it also in our case.
\begin{proposition}
There exists a constant $C(d)>0$ such that the following holds true:
for every $v \in L^2(\mathbb{R}^d)$, for every $h \in \mathbb{R}^d$ and for every open bounded set $ \Omega' \subset\mathbb{R}^d$ we have
	\begin{equation}\label{prop1formenunc}
	\| \tau_{h}v-v \|_{L^2(\Omega')}^{2} \leq C(d)\frac{\vert h \vert^{2}}{\rho^{d+2}}\int_{B_\rho}\|\tau_{y}v-v \|_{L^2(\Omega'_{\vert h \vert})}^2\ud y\,\quad\textrm{for every }\rho\in(0,|h|],
	\end{equation}
	with $\Omega'_{\vert h \vert}$ defined as in \eqref{ingrass}.
\end{proposition}
\begin{proof}
The proof closely resembles the one of \cite[Proposition 5]{ADPM}.
	Let $\varphi \in \CC_{\cc}^1(B_1)$ be a fixed function with $\varphi\geq 0$ and $\int_{B_1} \varphi(x)\ud x=1$\,. For every $\rho>0$ we define the functions $U_\rho, V_\rho:\R^d\to \R$ as
\begin{equation*}
	U_\rho(x):=\frac{1}{\rho^d} \int_{B_\rho} v(x+y) \varphi\bigg(\frac{y}{\rho}\bigg)\ud y\,,\qquad V_\rho(x):=\frac{1}{ \rho^d}\int_{B_\rho}(v(x)-v(x+y))\varphi\bigg(\frac{y}{\rho}\bigg)\ud y\,;
\end{equation*}	
clearly, for every $\rho>0$ and for every $x\in\R^d$ 
	\begin{equation*}
	v(x)= U_\rho(x)+V_\rho(x)\,,
	\end{equation*}
	and hence
	\begin{equation}\label{aa1}
	|\tau_hv(x)-v(x)|^2\le 3|U_\rho(x+h)-U_\rho(x)|^2+3|V_\rho(x)|^2+3|V_\rho(x+h)|^2\,.
	\end{equation}
	By Jensen inequality, for every $\xi\in\R^d$ we have
	\begin{equation}\label{prop1form2dim}
	\vert V_\rho(\xi) \vert^{2} \leq \frac{\omega_d}{\rho^d}\|\varphi \|^2_{L^{\infty}(B_1)} \int_{B_\rho} \vert v(\xi)-\tau_{y}v(\xi) \vert^2 \ud y.
	\end{equation}
	Moreover, by the change of variable $z=x+y$, we have that
	$$
	U_\rho(x)=\frac{1}{\rho^d}\int_{B_\rho(x)}v(z)\ffi\bigg(\frac{z-x}{\rho}\bigg)\ud z\,,
	$$
	whence we deduce that
	\begin{equation*}
	\begin{split}
	\mathrm{D}\,U_\rho(x)=& -\frac{1}{\rho^{d+1}}\int_{B_\rho(x)} v(z)\mathrm{D}\varphi\bigg(\frac{z-x}{\rho}\bigg)\ud z \\
	=&-\frac{1}{\rho^{d+1}}\int_{B_\rho(x)}(v(z)-v(x))\mathrm{D}\varphi\bigg(\frac{z-x}{\rho}\bigg)\ud z\\
	=&-\frac{1}{\rho^{d+1}}\int_{B_\rho}(v(x+y)-v(x))\mathrm{D}\varphi\bigg(\frac{y}{\rho}\bigg)\ud y\,;
	\end{split}
	\end{equation*}
	therefore, by the fundamental Theorem of Calculus and by Jensen inequality, we obtain 
	\begin{equation}\label{prop1form3dim}
	\begin{split}
	& \vert U_\rho(x+h)- U_\rho(x)\vert^{2} \leq \vert h \vert^{2} \int_{0}^{1} \vert \mathrm{D}\,U_\rho(x+th) \vert^2 \ud t \\
	& \leq \omega_d\frac{\vert h \vert^2}{\rho^{d+2}}\| \mathrm{D}\varphi \|_{L^{\infty}(B_1)}^2\int_{0}^{1} \int_{B_\rho} \vert \tau_yv(x+th)-v(x+th)\vert^2 \ud y\ud t.
	\end{split}
	\end{equation}
	Now, by \eqref{aa1}, \eqref{prop1form2dim}, and \eqref{prop1form3dim}, taking $\rho<|h|$\,, we have
	\begin{equation}\label{aa2}
	\begin{split}
	\vert \tau_hv(x)-v(x)\vert^2 \leq  & 3 \omega_d\frac{\vert h \vert^2}{\rho^{d+2}} \| \mathrm{D} \varphi \|_{\infty}^2 \int_{0}^{1} \int_{B_\rho} \vert \tau_yv(x+th)-v(x+th)\vert^2 \ud y \ud t\\
	&+ 3\omega_d\frac{|h|^2}{\rho^{d+2}}\|\varphi \|_{\infty}^2 \int_{B_\rho}  \vert \tau_yv(x)-v(x)\vert^2 \ud y \\
	&+3\omega_d \frac{\vert h \vert^2}{\rho^{d+2}}\|\varphi \|_{\infty}^2\int_{B_\rho} \vert \tau_y v(x+h)-v(x+h)\vert^{2}\ud y\,.
	\end{split}
	\end{equation}
	Finally, by integrating \eqref{aa2} on $\Omega'$\,, by Fubini theorem, we obtain \eqref{prop1formenunc} with $C(d):=3\omega_d(2 \| \varphi \|_{L^{\infty}(B_1)}^{2}+ \|\mathrm{D} \varphi \|_{L^{\infty}(B_1)}^2)$\,.
\end{proof}
We recall the following version of Hardy's inequality, that is proven in \cite[Proposition 6]{ADPM}.
\begin{lemma}\label{Hardytipodis}
	Let $g: \mathbb{R}\rightarrow [0,+\infty)$ be a Borel measurable function. Then for all $l \geq 0$ we have 
	\begin{equation*}
	\int_{0}^{r} \frac{1}{\rho^{d+l+1}} \int_{0}^{\rho} g(t)\ud t\ud \rho \leq \frac{1}{d+l}\int_{0}^{r} \frac{g(t)}{t^{d+l}}\ud t \quad \textrm{for every } r\geq0.
	\end{equation*}
\end{lemma} 
The following result will be used in the proof of Theorem \ref{thmGammaconvs1}(i). It is the $L^2$ analog of \cite[Proposition 4]{ADPM}.
\begin{proposition}\label{prop2enunciato}
	There exists a constant $\bar{C}(d)>0$ such that for every $v \in L^2(\mathbb{R}^d)$, for every bounded open set $\Omega' \subset \mathbb{R}^d$, for every $s \in (0,1)$\,, and for every $h \in \mathbb{R}^d$\,, it holds 
	\begin{equation*}
	\| \tau_{h}v-v \|_{L^2(\Omega')}^{2} \leq \vert h \vert^{2s}\bar{C}(d)(1-s) \int_{B_{\vert h \vert}} \frac{\|\tau_{y}v-v \|_{L^{2}(\Omega'_{\vert h \vert})}^{2}}{\vert y \vert^{d+2s}}\ud y\,.
	\end{equation*}
\end{proposition}
\begin{proof}  For every fixed $v \in L^2(\mathbb{R}^d)$, 
	we define the function $g_v:[0,|h|]\to\R$ as 
	$$
	g_v(t):= \int_{\partial B_t} \|\tau_{y} v - v  \|_{L^{2}(\Omega'_{\vert h \vert })}^{2}       \ud \mathcal{H}^{d-1}(y)\,.
	$$
	By integrating in polar coordinates formula \eqref{prop1formenunc} we have
	\begin{equation}\label{prop2fomr1dim}
	\| \tau_{h}v-v \|_{L^{2}(\Omega')}^{2} \leq C(d)\frac{\vert h \vert^{2}}{\rho^{d+2}} \int_{0}^{\rho} g_v(t)\ud t\,.
	\end{equation}
	By multiplying both sides of \eqref{prop2fomr1dim} by $\rho^{1-2s}$ and 
 integrating in the interval $ [0,\vert h \vert]$\,, using Lemma \ref{Hardytipodis}  and the very definition of $g_v$\,,
 we obtain
	\begin{equation*}
	\begin{aligned}
	\| \tau_{h}v-v \|_{L^{2}(\Omega')}^{2} \leq& 2C(d)(1-s) \vert h \vert^{2s} \int_{0}^{\vert h \vert } \frac{1}{\rho^{d+2s+1}} \int_{0}^{\rho} g_v(t)\ud t\ud z\\
	\le&C(d)(1-s) \vert h \vert^{2s}\int_{0}^{|h|}\frac{g_v(t)}{t^{d+2s}}\ud t\\
	=&C(d)(1-s) \vert h \vert^{2s} \int_{B_{\vert h \vert}} \frac{\|\tau_{y}v-v \|_{L^{2}(\Omega'_{\vert h \vert})}^{2}}{\vert y \vert^{d+2s}}\ud y\,,
	\end{aligned}
	\end{equation*}
	 which concludes the proof. 
	\end{proof}

We are now in a position to prove Theorem \ref{thmGammaconvs1}(i).
\begin{proof}[Proof of Theorem \ref{thmGammaconvs1}(i)]
	By 
	Proposition \ref{prop2enunciato} and by the upper bound \eqref{COMP3} we obtain that for  every open bounded set $\Omega'\subset \mathbb{R}^d$ and for every $h\in\R^d$
	\begin{equation} \label{teocompform1}
	\| \tau_{h}\tilde{u}^n-\tilde{u}^n \|_{L^{2}(\Omega')} \leq C(d,M)\vert h \vert^{s_n}\,,
	\end{equation}
	where we recall that $\tilde u^n$ is the extension of $u^n$ to $0$ in  $\mathbb{R}^d \setminus \Omega$\,.	
	Therefore, the sequence $\{\tilde{u}^n\}_{n \in \mathbb{N}}$ satisfies the assumption of Theorem \ref{frechetkolmthm}, and hence there exists a function $v \in L^{2}(\mathbb{R}^d)$ with $v=0 $ in $\mathbb{R}^d \setminus \Omega$, such that, up to a subsequence, $\tilde u^n \rightarrow  v $ in $L_{\mathrm{loc}}^2(\mathbb{R}^d)$. Now, sending $n\to +\infty$ in \eqref{teocompform1}, we obtain that for every open bounded set $\Omega' \subset \mathbb{R}^d$
	\begin{equation*} 
	\| \tau_{h}v-v \|_{L^{2}(\Omega')} \leq C(d,M)\vert h \vert \quad \textrm{for every }h \in \mathbb{R}^d \,,
	\end{equation*}
	and hence by Theorem \ref{carsobvshift} (choosing $\Omega'=\R^d$ in the above inequality) we obtain that $\mathrm{D}v \in L^2(\mathbb{R}^d)$\,. Since $v=0$ in $\mathbb{R}^d \setminus \Omega$, by the regularity of $\partial \Omega$, we have that $v$ is the extension to $0$ in $\R^d\setminus\Omega$ of a function $u \in H_0^1(\Omega)$\,, thus concluding the proof.
\end{proof}
\subsection{Proof of the $\Gamma$-liminf inequality}
Here we prove the $\Gamma$-liminf inequality in Theorem~\ref{thmGammaconvs1}. The proof of this result closely resembles the one of \cite[Theorem 2]{BBM1}.
\begin{proof}[Proof of Theorem \ref{thmGammaconvs1}(ii)]
We can assume without loss of generality that \eqref{COMP3} holds true so that
the function $u$ is actually in $H^1_0(\Omega)$.

{\it Claim 1.} {\it Let $\eta\in \CC_{\cc}^\infty(B_1)$ be a standard mollifier, i.e., $\eta\ge 0$ and $\int_{B_1}\eta(x)\ud x=1$\,. For every $\ep>0$\,, we set  $\eta_\ep(\cdot):=\frac{1}{\ep^d}\eta(\frac{\cdot}{\ep})$\,.
Then, for every $s\in (0,1)$
\begin{equation*}
 \frac 1 2[\tilde{v}_\ep]^2_s \leq F^s(v)\qquad\textrm{for every }v\in L^2(\Omega)\textrm{ and for every }\ep>0\,,
\end{equation*}
where $\tilde{v}_\ep:=\tilde{v}\ast \eta_{\varepsilon}$\,.}

Indeed, setting $\Omega_{\varepsilon}:=\{ x\in\R^d\ :\ \text{dist}(x,\Omega)\leq\varepsilon\}$\,, we have that $\tilde v_\ep=0$ in $\R^d\setminus\Omega_\ep$\,; 
therefore, by applying Jensen inequality to the probability measure $\eta_{\varepsilon}\ud z$\,, we get
\begin{equation*}
\begin{aligned}
[\tilde{v}_{\varepsilon}]^2_s\leq& \int_{\R^d}\int_{\R^d} \int_{\R^d} \frac{|\tilde{v}(x-z)-\tilde{v}(y-z)|^2}{|x-y|^{d+2s}} \eta_{\varepsilon}(z)\ud z \ud y\ud x\\
=& \int_{\R^d}\int_{\R^d} \int_{\R^d}\frac{|\tilde{v}(x-z)-\tilde{v}(y-z)|^2}{|x-z-(y-z)|^{d+2s}} \eta_{\varepsilon}(z)\ud z \ud y\ud x\\
 =&2F^s(v)\,. 
 \end{aligned}
 \end{equation*}
{\it Claim 2.} {\it For every $\ep>0$ and for every $R>0$, it holds
\begin{equation}\label{stepdue}
\frac{\omega_d}{4}\liminf_{n\to +\infty}\,\int_{B_R}|\nabla \tilde u^n_\ep|^2(\mathrm{dist}(x,\partial B_R))^{2(1-s_n)}\ud x\le \liminf_{n\to +\infty}(1-s_n)\frac{[\tilde{u}^n_\ep]^2_{s_n}}{2}\,, 
\end{equation}
where $\tilde u_\ep^n:=\tilde u^n\ast \eta_\ep$\,, with $\eta_\ep$ as in Claim 1.}

Indeed, by Taylor expansion, using that  $\sup_{n\in\N}\|u^n\|^2_{L^2(\Omega)}\le M$ we have that
\begin{equation*}
\begin{aligned}
|\tilde{u}_\ep^{n}(x)-\tilde{u}_\ep^{n}(y)|^2\ge&\Big|\nabla \tilde{u}_\ep^{n}(x)\cdot\frac{x-y}{|x-y|}\Big|^2|x-y|^2\\
&-\|u^n\|^2_{L^1(\Omega)}\|\eta_\ep\|^2_{C^{2}(\R^d)}|x-y|^3-\|u^n\|_{L^1(\Omega)}^2\|\eta_\ep\|^2_{C^{2}(\R^d)}|x-y|^4\\
\ge&\Big|\nabla \tilde{u}_\ep^{n}(x)\cdot\frac{x-y}{|x-y|}\Big|^2|x-y|^2-C(\ep,M)(|x-y|^3+|x-y|^4)\,,
\end{aligned}
\end{equation*}
Therefore, for every $x\in\R^d$\,, setting $\delta:=\mathrm{dist}(x,\partial B_R)$\,, we get
\begin{equation}\label{term}
\begin{aligned}
&(1-s_n) \int_{B_R} \frac{|\tilde{u}_\ep^{n}(x)-\tilde{u}_\ep^{n}(y)|^2}{|x-y|^{d+2s_n}}\ud y\geq  (1-s_n)\int_{B_\delta(x)} \frac{|\tilde{u}_{\varepsilon}^n(x)-\tilde{u}_{\varepsilon}^n(y)|^2}{|x-y|^{d+2s_n}}\ud y\\
\ge&(1-s_n)\int_{B_\delta(x)}\Big|\nabla \tilde{u}_\ep^{n}(x)\cdot\frac{x-y}{|x-y|}\Big|^2|x-y|^{2(1-s_n)-d}\ud y\\
&-(1-s_n)C(\ep,M)\int_{B_\delta(x)}\frac{|x-y|^{3}+|x-y|^4}{|x-y|^{d+2s_n}}\ud y\\
=&\frac{\omega_d}{2}\delta^{2(1-s_n)}|\nabla \tilde u_\ep^n(x)|^2-(1-s_n)C(\ep,M,d)\,,
 \end{aligned}
\end{equation}
where in the last equality we  integrated over spherical boundaries from 0 to $\delta$, using  that $\int_{\mathbb{S}^{d-1}}|\nabla  \tilde{u}_\ep^{n}(x)\cdot\theta|^2\ud \theta=\omega_d|\nabla   \tilde{u}_\ep^{n}(x)|^2$\,.
By integrating \eqref{term} over $B_R$, we get \eqref{stepdue}.

By Claim 1 and Claim 2, for every $\ep>0$ and for every $R>0$ we have that
\begin{equation*}
\liminf_{n\to +\infty}(1-s_n)F^{s_n}(u^n)\geq  \frac{\omega_d}{4}\liminf_{n\to +\infty} \int_{B_R}|\nabla \tilde{u}_{\varepsilon}^n(x)|^2(\mathrm{dist}(x,\partial B_R))^{2(1-s_n)}\ud x\,,
\end{equation*}
whence, using that for every $\ep>0$ the sequence $\{\tilde u_\ep^n\}_{n\in\N}$ is equi-Lipschitz, 
 we get that, up to a (not relabeled) subsequence,
\begin{equation}\label{PreLinf'} 
\liminf_{n\to +\infty}(1-s_n)F^{s_n}(u^n)\geq  \frac{\omega_d}{4}\liminf_{n\to +\infty} \int_{B_R}|\nabla \tilde{u}_{\varepsilon}^n(x)|^2\ud x\,.
\end{equation}
Notice that $\tilde u^n_\ep\to \tilde u_\ep:=\tilde u\ast\eta_\ep$ in $L^2(\R^d)$ as $n\to +\infty$\,.
By \eqref{PreLinf'} and by \eqref{COMP3}, we have that in fact $\tilde u^n_\ep\weakly \tilde u_\ep$ in $H^1(B_R)$\,.  

In conclusion, by \eqref{PreLinf'}, we deduce that for every $\ep>0$
$$
\liminf_{n\to +\infty}(1-s_n)F^{s_n}(u^n)\ge  \frac{\omega_d}{4}\int_{B_R}|\nabla \tilde{u}_{\varepsilon}(x)|^2\ud x\,,
$$
whence the claim follows sending first $\ep\to 0$ (using that $\tilde u_\ep\to \tilde u$ in $H^1(\R^d)$ as $\ep\to 0$)
and then $R\to +\infty$\,. 
\end{proof}
\subsection{Proof of the $\Gamma$-limsup inequality}
The proof of the $\Gamma$-limsup inequality relies on the pointwise convergence of $(1-s)F^s$ to $F^1$ (as $s\to 1$) for smooth functions with compact support and on the density of smooth functions in $H^1_0(\Omega)$\,.
As for the pointwise convergence we recall the following result, proved in \cite{Monik} in a more general setting.
\begin{theorem}\label{teocitaziones1limsup}
	For every $v \in \CC_\cc^{\infty}(\mathbb{R}^d)$ it holds
	$$
	 \lim_{s \rightarrow 1^-} (1-s) \int_{\mathbb{R}^d} \int_{\mathbb{R}^d} \frac{\vert v(x)-v(y) \vert^2}{\vert x-y \vert^{d+2s}}\ud x\ud y= \frac{\omega_d}{2}\int_{\R^d}|\nabla v(x)|^2\ud x\,.
	$$ 
\end{theorem}
With Theorem \ref{teocitaziones1limsup} on hand we can prove Theorem \ref{thmGammaconvs1}(iii) using standard density arguments in $\Gamma$-convergence.
\begin{proof}[Proof of Theorem \ref{thmGammaconvs1}(iii)]
It is enough to prove the claim only for $u\in H^1_0(\Omega)$\,.
For every $u \in H_0^1(\Omega)$ there exists $\{u^k\}_{k \in \mathbb{N}} \subset \CC_\cc^{\infty}(\Omega)$ such that $ u^{k} \rightarrow u$ (as $k \rightarrow + \infty$) in $H^1(\Omega)$\,. In view of Theorem \ref{teocitaziones1limsup} we have that for every $k \in \mathbb{N}$
\begin{equation*}
\begin{aligned} 
\lim_{ n \rightarrow +\infty} (1-s_n) F^{s_n}(u^k)=& \lim_{n \rightarrow +\infty} \frac{1-s_n}{2} \int_{\mathbb{R}^d} \int_{\mathbb{R}^d} \frac{\vert \tilde u^k(x)-\tilde u^k(y) \vert^2}{\vert x-y \vert^{d+2s_n}}\ud x\ud y\\
=&  \frac{\omega_{d}}{4} \int_{\R^d} \vert \nabla \tilde u^k(x) \vert^2 \ud x= \frac{\omega_{d}}{4} \int_{\Omega} \vert \nabla u^k(x) \vert^2 \ud x\,.
\end{aligned}
\end{equation*}
Therefore by a standard diagonal argument there exists $\{k_n\}_{n \in \mathbb{N}}$ such that
$$ 
\lim_{n \rightarrow + \infty} u^{k_n}=u, \quad \limsup_{n \rightarrow + \infty}(1-s_n)F^{s_n}(u^{k_n}) \leq \frac{\omega_{d}}{4} \int_{\Omega} \vert \nabla u (x)\vert^2\ud x= F^1(u)\,,
$$
i.e., \eqref{lsupsto1}.
\end{proof}
\section{Minimizing movements for $\lambda$-convex functionals defined on a Hilbert space}\label{sec:general}
In this section we develop the general theory that will allow us to study the stability of the $s$-fractional heat flow as $s\to 0^+$ and $s\to 1^-$\,. 
Throughout this section $\h$ is a generic Hilbert space, $\langle\cdot, \cdot \rangle_\h$ is the inner product of $\h$ and $\vert \cdot \vert_\h$ is the norm induced by such a scalar product.  In the abstract setting of this section, we denote by $\dot{v}$ the time derivative of any function $v$ from a time interval with values in $\h$. 
\begin{definition}[$\lambda$-convexity,  $\lambda$-positivity, $\lambda$-coercivity]
Let $\lambda>0$\,.	We say that a function $\F: \h\rightarrow (-\infty, +\infty]$ is $\lambda$-convex if the function $f(\cdot)+ \frac{\lambda}{2} \vert \cdot \vert_\h^2$ is convex.
Moreover, we say that $\F$ is $\lambda$-positive if 
$\F(x)+ \frac{\lambda}{2} \vert x \vert_\h^2 \geq 0$ for every $x \in\h$\,,
and we say that $\F$ is $\lambda$-coercive if the sublevels of the function  $\F(\cdot)+ \frac{\lambda}{2} \vert \cdot \vert_\h^2$ are bounded.
\end{definition}
\begin{remark}\label{implicata}
\rm{We notice that if $\F$ is $\lambda$-positive, then $\F$ is $\tilde\lambda$-coercive for every $\tilde\lambda>\lambda$\,. 
}
\end{remark}
\begin{proposition}\label{exisitminprobincre}
	Let $\F:\h\rightarrow (-\infty,+\infty]$ be a proper, strongly lower semicontinuous function which is $\lambda$-convex and $\lambda$-positive for some $\lambda>0$\,. Then for every $0<\tau < \frac{1}{ 2 \lambda}$ and for every $y \in \h$ the problem 
\begin{equation}\label{schemaMMtau}
\min \biggl\{ \F(x)+ \frac{1}{2 \tau} \vert x-y \vert_\h^2: \; x \in \h\biggr\}
\end{equation}
admits a unique solution.
\end{proposition}
\begin{proof}
 We preliminarily notice that, since $\F$ is $\lambda$-convex and strongly lower semicontinuous, then the function $\F(\cdot)+\frac{1}{2\tau}|\cdot|^2_\h$ is strictly convex and strongly lower semicontinuous and, in turn, weakly lower semicontinuous.  Clearly, this implies that also $\F(\cdot)+\frac{1}{2\tau}|\cdot-y|^2_\h$ is weakly lower semicontinuous.  
 Moreover, by Remark \ref{implicata}, we have that $\F$ is $\frac{1}{2\tau}$-coercive.

Since $\F$ is proper,
$$
0\le \inf\biggl\{ \F(x)+ \frac{1}{2 \tau} \vert x-y \vert_\h^2: \; x \in \h\biggr\}\le M\,,
$$ 
for some $M>0$\,.
Let $ \{x_k\}_{k \in\mathbb{N}} \subset \h$ be a sequence such that 
 \begin{equation}\label{limcheminimizzaprop1} 
 \lim_{k \rightarrow +\infty} \F(x_k)+\frac{1}{2  \tau} \vert x_k-y \vert_\h^2= \inf \biggl\{ \F(x)+ \frac{1}{2 \tau} \vert x-y \vert_\h^2: \; x \in \h\biggr\}\,.  
 \end{equation}
By triangle inequality,  for $k$ sufficiently large, we have
 \begin{equation*}
 \begin{aligned}
 2M\ge&\F(x_k)+ \frac{1}{2 \tau} \vert x_k-y \vert_\h^2 \geq \F(x_k)+ \frac{1}{4 \tau} \vert x_k \vert_\h^2 - \frac{1}{2 \tau}\vert y \vert_\h^2 
 \end{aligned}
 \end{equation*}
whence, in view of the $\frac{1}{2\tau}$-coercivity of the function $\F$\,, we deduce that, up to a subsequence, $x_k\overset{\h}{\weakly} x_\infty$ for some $x_\infty\in\h$\,.
Therefore, by \eqref{limcheminimizzaprop1} and by the weak lower semicontinuity of the function $\F(\cdot)+\frac {1}{2\tau}|\cdot-y|^2_\h$\,,  we obtain 
 \begin{multline*}
 \inf \biggl\{ \F(x)+ \frac{1}{2 \tau} \vert x-y \vert_\h^2: \; x \in \h\biggr\} = \lim_{k \rightarrow +\infty} \F(x_k)+ \frac{1}{2 \tau} \vert x_k -y \vert_\h^2 \\ 
 \ge \F(x_{\infty})+ \frac{1}{2 \tau}\vert x_{\infty}-y \vert_\h^2\,,
 \end{multline*}
 i.e., that $x_\infty$ is a minimizer of the problem in \eqref{schemaMMtau}.
 
 Finally, the uniqueness of the solution is a consequence of the strict convexity of the functional
 $\F(\cdot)+\frac{1}{2\tau}|\cdot-y|^2$\,.
\end{proof}
For every function $\F: \h \rightarrow (-\infty,+\infty]$ we denote by $D(\F)$ the set of all $x \in \h$ such that $\F(x)\in\R$.
\begin{definition}[Fr\'echet subdifferential] \label{def:subdif}
	For $\F: \h\rightarrow (-\infty,+\infty]$ and $ x \in D(\F)$, the Fr\'echet subdifferential of $\F$ at $x$ is defined as
\begin{equation*}
	\partial \F(x):= \biggl\{ v \in \h: \liminf_{y \to x} \frac{ \F(y)-\F(x)-\langle v, y-x \rangle_\h}{\vert y -x \vert_\h} \geq 0\biggl\}\,.
	\end{equation*} 
\end{definition} 
\begin{remark}\label{rem:subdiflambdaconvex}
Whenever $\F$ is a $\lambda$-convex function it holds that
\begin{equation}\label{subdiflambdaconvex}
\partial \F(x)=\biggl\{ v \in \h: \, \F(y)-\F(x)-\langle v, y-x \rangle_\h \geq - \frac{\lambda}{2}\, \vert y -x \vert_\h^2 \quad \text{for every } y \in \h \biggr\}\,.
\end{equation}
Indeed, for a convex function $\phi$, $v \in \partial \phi(x)$ if and only if $\phi(y)-\phi(x)-\langle v, y-x \rangle_\h \geq 0$ for every $y \in \h$, namely the Fréchet subdifferential coincides with the usual subdifferential of convex analysis. 
Then, being $\F$ $\lambda$-convex and since $\partial \Big(\phi + \frac{\lambda}{2} |\cdot|_\h^2\Big)= \partial\phi +  \lambda\, \cdot$, it holds that $v \in \partial \F(x)$ if and only if 
\[
\F(y) + \frac{\lambda}{2} |y|^2_\h-\F(x) - \frac{\lambda}{2} |x|^2_\h-\langle v+\lambda x, y-x \rangle_\h \geq 0 \quad\text{for every }y \in \h\,,
\]
which coincides with the condition in \eqref{subdiflambdaconvex} since $|y-x|^2_\h=|y|^2_\h-|x|^2_\h - 2\langle x, y-x \rangle_\h$.
\end{remark}
Let $\F : \h \rightarrow (-\infty,+\infty]$ be a proper, strongly lower semicontinuous function which is $\lambda$-positive and $\lambda$-convex, for some $\lambda>0$\,, and let $x_0\in D(\F)$\,. For every $0<\tau<\frac{1}{2\lambda}$, we denote by $\{x_k^\tau\}_{k\in\N}$ the discrete-in-time evolution for $\F$ with initial datum $x_0$\,, defined by
\begin{equation}\label{movmin}
x_0^\tau:=x_0\,, \qquad x_{k+1}^\tau \in \arg\!\min \biggl\{ \F(x)+ \frac{1}{2 \tau} \vert x-x_k^{\tau}\vert^2_\h \biggr\} \quad \textrm{for every } k \in \mathbb{N}\cup\{0\}\,.
\end{equation}
Since $x_0\in D(\F)$\,, then $x_k^\tau\in D(\F)$ for every $k\in\N$\,.
Furthermore, we define the piecewise-affine interpolation $x^\tau:[0,+\infty)\rightarrow \h$ of $\{x^\tau_k\}_{k\in\N}$ as 
\begin{equation}\label{funzdeftratti}
x^{\tau}(t):= x_{k}^{\tau}+ \frac{x_{k+1}^{\tau}-x_{k}^{\tau}}{\tau}(t-k\tau), \quad t \in [k\tau, (k+1)\tau).
\end{equation}
\begin{theorem}\label{exist}
	Let $\F: \h \rightarrow (-\infty,+\infty]$ be a proper, strongly lower semicontinuous function  which is $\lambda$-convex and $\lambda$-positive, for some $\lambda>0$\,. Let moreover $ x_0 \in D(\F)$\,. Then, there exists  a unique solution $x \in H^1([0,+\infty);\h)$ to the following Cauchy problem
	\begin{equation}\label{probcauchy}
	\begin{dcases}
		 \dot{x}(t)  \in - \partial \F(x(t)) \quad  \text{for a.e.\ $t \in [0,+\infty)$},  \\
		x(0)=x_0\,.
	\end{dcases} 
	\end{equation} 
	Moreover, for every $T>0$, $x^\tau\weakly x$ in $H^1([0,T];\h)$\,,
where $x^{\tau} $ is defined in  \eqref{funzdeftratti} for $0<\tau<\frac{1}{2\lambda}$\,.  
Furthermore, 
\begin{eqnarray}\label{dopo0}
&&\| \dot{x} \|^2_{L^2((0,T);\h)}\le 4^{8\lambda T+4}(|\F(x_0)|+\lambda|x_0|_\h^2)\textrm{ for every }T>0\,,\\
\label{dopo}
&&|x(t)-x^\tau(t)|^2_\h\le  8 \tau\, 12^{8\lambda t+4}\big(|\F(x_0)|+\lambda|x_0|_\h^2\big)
\textrm{ for every }t\ge0\,, \,\tau<\frac{1}{16\lambda}\,.
\end{eqnarray}
\end{theorem}
\begin{proof}
	\textit{ Uniqueness.}
	Let $T>0$ and let $x_1, x_2 \in H^1([0,T];\h)$ satisfy the Cauchy problem \eqref{probcauchy}  up to time $T$. 
    We first observe that
    \begin{equation}\label{monsotdiff}
    \langle y_1-y_2, v_1-v_2 \rangle_\h \leq  \lambda\,  \vert y_1 - y_2 \vert_\h^2\quad\textrm{for every }y_1,y_2\in\h\,, \,-v_i \in \partial\F(y_i)\textrm{ for }i=1,2\,.
    \end{equation}
    Indeed, by  \eqref{subdiflambdaconvex},  we have 
    $$
     \F(y)-\F(y_1)+ \langle v_1, y-y_1 \rangle_\h \geq - \frac{\lambda}{2}\,  \vert y-y_1 \vert_\h^2\,, \quad y \in\h\,,
     $$
which, for $y=y_2$ implies
    \begin{equation}\label{formulaasd123123}
    \F(y_2)-\F(y_1)+ \langle v_1, y_2-y_1 \rangle_\h \geq - \frac{\lambda}{2}\, \vert y_2-y_1 \vert_\h^2;
    \end{equation}
   analogously
    \begin{equation}\label{formulaasd12341234}
   \F(y_1)-\F(y_2)+ \langle v_2, y_1-y_2 \rangle_\h \geq - \frac{\lambda}{2}\, \vert y_2-y_1 \vert_\h^2\,.
    \end{equation}
   Therefore, \eqref{monsotdiff}  follows by summing \eqref{formulaasd123123} and \eqref{formulaasd12341234}. 
   
   Finally, by formula \eqref{monsotdiff} we have 
    $$ \frac{\ud}{\ud t} \vert x_1(t)-x_2(t) \vert_\h^2 = 2 \langle \dot{x}_1(t)-\dot{x}_2(t), x_1(t)-x_2(t) \rangle_\h\leq  2  \lambda\, \vert x_1(t)-x_2(t) \vert_\h^2 \quad \text{  for  a.e. $t \in [0,T]$,}\, $$
    which, by Gronwall's Lemma, implies 
    $$ \vert x_1(t)-x_2(t)\vert_\h^2 \leq \vert x_0 -x_0 \vert_\h e^{ 2  \lambda t}=0 \quad \text{  for  a.e. $t \in [0,T]$\,,} $$
    i.e., $x_1(t)= x_2(t)$ for a.e. $t \in [0,T]$.  We notice that the solution is in $C^{0,\frac12}([0,T];\h)$ by the Sobolev embedding of $H^1([0,T];\h)$ into $C^{0,\frac12}([0,T];\h)$, so that $x_1(t)= x_2(t)$ when passing to the continuous representatives. 
    \vskip5pt
    \textit{ Existence.} 
    We first prove that for every $T>0$ the functions $x^\tau$ defined in \eqref{funzdeftratti} converge (as $\tau\to 0$) weakly in $H^1([0,T];\h)$ to some function $x\in H^1([0,T];\h)$
    and then we show that the limit $x$ satisfies \eqref{probcauchy}  up to time $T$.   
    
By \eqref{movmin} we have that 
    \begin{equation}\label{minimalitàthmmain}
   \F(x_{k+1}^{\tau})+ \frac{1}{2 \tau} \vert x_{k+1}^{\tau}-x_{k}^{\tau}  \vert_\h^2 \leq \F(x_k^\tau), \quad \text{for every $k \in \mathbb{N}$\,,}
   \end{equation}
   which together with the $\lambda$-positivity of $\F$ implies that
   \begin{equation}\label{useful}
   \begin{aligned}
&    \sum_{k=0}^{K} \frac{1}{\tau} \left \vert x^{\tau}_{k+1}-x^{\tau}_k   \right \vert_\h^2   \leq2\sum_{k=0}^{K}\left  ( \F(x^{\tau}_k)-\F(x^{\tau}_{k+1}) \right)    \\
   =& 2(\F(x^{\tau}_0)-\F(x^{\tau}_{K+1})) \\
  =& 2\Big(\F(x^{\tau}_0)+ \frac{\lambda}{2}\vert x^{\tau}_{K+1}\vert_\h^2 - \frac{\lambda}{2}\vert x^{\tau}_{K+1}\vert_\h^2-  \F(x^{\tau}_{K+1})\Big) \\
  \leq &2\left ( \F(x_0)+ \frac{\lambda}{2} \vert x^{\tau}_{K+1}\vert_\h^2\right)\qquad\textrm{for every }K\in\N\,.
  \end{aligned}
   \end{equation}
   Set $\hat T=\frac{1}{8\lambda}$ and let $0<\tau\le\frac{1}{16\lambda}$\,. 
   We set $\hat K:= \left \lceil \frac{\hat T}{\tau} \right \rceil $\,; by \eqref{useful},
   we have
   \begin{equation}\label{bound1}
   \begin{split}
   \int_{0}^{\hat T}  \vert \dot x^{\tau}(t) \vert_\h^2 \ud t  \le&  \sum_{k=0}^{\hat K} \frac{1}{\tau} \left \vert x^{\tau}_{k+1}-x^{\tau}_k   \right \vert_\h^2 
  \leq 
  2\left ( \F(x_0)+ \frac{\lambda}{2} \vert x^{\tau}_{\hat K+1}\vert_\h^2\right)\,.
   \end{split}
\end{equation}
Moreover, by triangle and Jensen inequalities and using again \eqref{useful}, we get 
\begin{equation*}
\begin{aligned}
&\frac{1}{2}|x_{\hat K+1}^\tau|_\h^2-|x_0|_\h^2\le |x_{\hat K+1}^\tau-x_0|^2_\h\\
 \leq& \tau(\hat K+1)\sum_{k=1}^{\hat K+1} \frac{1}{\tau}\left\vert x^{\tau}_k- x^{\tau}_{k-1} \right\vert_\h^2 
 =\tau(\hat K+1)\sum_{k=0}^{\hat K} \frac{1}{\tau}\left\vert x^{\tau}_{k+1}- x^{\tau}_{k} \right\vert_\h^2\\
 \le& 2(\hat T+2\tau)\Big(\F(x_0)+\frac{\lambda}{2}|x_{\hat K+1}^\tau|_\h^2\Big)\,,
\end{aligned}
\end{equation*}
which, recalling that $0<2\tau\le\frac{1}{8\lambda}=\hat T$\,, implies that
\begin{equation}\label{bound3}
\vert x^{\tau}_{\hat K+1}\vert^2_{ \h } \leq \frac{2}{\lambda}\F(x_0)+4|x_0|^2_\h\,.
\end{equation}
By \eqref{bound1} and \eqref{bound3}, we have that, for every $\tau$ small enough, 
\begin{equation}\label{bound10}
\|\dot{x}^\tau\|^2_{L^2((0,\hat{T});\h)}\le 4(\F(x_0)+\lambda|x_0|_\h^2)\,.
\end{equation}
Iterating the estimates in \eqref{bound3} and \eqref{bound10}, using also that $\F(x^\tau_k)$ is not increasing with respect to $k$, we deduce that for every $j\in\N$
\begin{equation*}
\begin{split}
|x^\tau(j\hat{T})|_\h^2& \le 4^j\Big(\frac{1}{\lambda} |\F(x_0)|+|x_0|_\h^2\Big),\\
\|\dot{x}^\tau\|^2_{L^2((0,j\hat T);\h)}& \le 4^{j+3}(|\F(x_0)|+\lambda|x_0|_\h^2)\,.
\end{split}
\end{equation*}
In particular, for every $T>0$, we have that
\begin{equation}\label{bound100}
\|\dot{x}^\tau\|^2_{L^2((0,T);\h)}\le 4^{8\lambda T+4}(|\F(x_0)|+\lambda|x_0|_\h^2)\,.
\end{equation}
Therefore, for every $T>0$, $\|x^\tau\|_{H^1([0,T];\h)}$ is uniformly bounded and hence, up to a subsequence, $x^\tau\weakly x$ in $H^1([0,T];\h)$ for some $x\in H^1([0,T];\h)$\,; this, in particular, implies the convergence in $C^{0,\frac{1}{2}}([0,T];\h)$ and hence that $x(0)=x_0$\,.
 Passing to the limit in \eqref{bound100} we readily get \eqref{dopo0}.

\vskip5pt
Now we aim at proving that $x$ solves \eqref{probcauchy}  up to time $T$, 
for every $T>0$\,, that is 
\begin{equation}\label{eqpc}
 \dot{x}(t) \in -\partial \F(x(t))\qquad\textrm{ for almost every }t\in (0,T)\,.
\end{equation}
To this end, we define the piecewise-constant interpolation $\tilde x^\tau:[0,+\infty)\to \h$ of $\{x_k^\tau\}_{k\in\N}$ as 
\begin{equation*}
\tilde x ^\tau(t):= x_{k+1}^\tau, \quad t \in [k\tau, (k+1)\tau)\,,
\end{equation*}
and we notice that, by minimality, for $\tau$ small enough, 
\begin{equation}\label{proprietà}
\dot{x}^\tau(t)\in -\partial \F(\tilde x^\tau(t))\qquad\textrm{for almost every }t\in [0,+\infty)\,.
\end{equation}
We claim that 
\begin{equation}\label{ctilde}
\tilde x^{\tau}\overset{\h}{\rightarrow} x\quad\textrm{ in }L^2((0,T);\h)\,,\quad\textrm{ for every }T>0\,.
\end{equation} 
Indeed, by triangle inequality, we have that
\begin{equation}\label{formasdasd111}
\begin{split}
 \| \tilde{x}^{\tau}-x \|^2_{L^2((0,T);\h)} & \leq 2\| x^{\tau}-x \|^2_{L^2((0,T);\h)} + 2\| \tilde{x}^{\tau}-x^{\tau} \|^2_{L^2((0,T);\h)}  \\
 & \leq 2\| x^{\tau}-x \|^2_{L^2((0,T);\h)} +2\tau^2 \|\dot{x}^\tau\|^2_{L^2((0,T);\h)}\,,
 \end{split}
\end{equation}
where in the last inequality we have used that 
$$ 
x^\tau(t)-\tilde{x}^{\tau}(t)= \frac{x_{k+1}^\tau- x_{k}^{\tau}}{\tau}(t-(k+1)\tau)=\dot{x}^\tau(t)(t-(k+1)\tau), \quad\textrm{for every } t \in (k\tau,(k+1)\tau)\,.
$$
Therefore,  by \eqref{formasdasd111}   and \eqref{dopo0} we get
\begin{equation*}
\| \tilde{x}^{\tau}-x \|_{L^2((0,T);\h)}^2 \leq
2\| x^{\tau}-x \|^2_{L^2((0,T);\h)} + 2\tau^2 4^{8\lambda T+4} (|\mathcal{F}(x_0)|+ \lambda \vert x_0 \vert_{\h}^2)\,,
\end{equation*} 
 which, sending $\tau\to 0$ and  recalling  that $x^\tau\weakly x$ in $H^1([0,T];\h)$, implies \eqref{ctilde}.
 
 With \eqref{ctilde} on hand, we are in a position to prove \eqref{eqpc}.
 Let $t_0 \in (0,T)$ be a Lebesgue point of the function $ \dot{x}:[0,T)\to \h$. By 
 \eqref{proprietà}, we have that
\begin{equation}\label{form123asdasd}
\F(y)\ge\F(\tilde{x}^{\tau}(t))-\langle \dot{x}^{\tau}(t), y- \tilde{x}^{\tau}(t) \rangle_{\h}- \frac{\lambda}{2} \vert y-\tilde{x}^{\tau}(t) \vert_{\h}^2\,\qquad\textrm{for every }y\in\h\,. 
\end{equation}
Let $y\in\h$ and $h>0$\,; by integrating \eqref{form123asdasd} in the interval $(t_0, t_0+h)$ and dividing by $h$\,, we obtain
 $$ 
 \F(y) \geq \frac{1}{h} \int_{t_0}^{t_0+h}\F(\tilde{x}^{\tau}(t))\ud t   -\frac{1}{h} \int_{t_0}^{t_0+h} \langle \dot{x}^{\tau}(t), y- \tilde{x}^{\tau}(t) \rangle_{\h}\ud t- \frac{1}{h} \int_{t_0}^{t_0+h} \frac{\lambda}{2} \vert y-\tilde{x}^{\tau}(t) \vert_{\h}^2 \ud t \,, 
 $$ 
which, sending $\tau\to0$\,, and using the strong lower semicontinuity of $\F$\,, the weak $L^2$-convergence of $\dot x^\tau$ to $\dot x$\,, and \eqref{ctilde}, yields
 $$ 
 \F(y) \geq \frac{1}{h} \int_{t_0}^{t_0+h}\F({x}(t))\ud t   -\frac{1}{h} \int_{t_0}^{t_0+h} \langle \dot{x}(t), y- {x}(t) \rangle_{\h}\ud t- \frac{1}{h} \int_{t_0}^{t_0+h} \frac{\lambda}{2} \vert y-{x}(t) \vert_{\h}^2 \ud t \,.
 $$ 
Now, since $x\in C^{0,\frac 1 2}([0,T];\h)$ and since $t_0$ is a Lebesgue point for $ \dot{x} $\,, sending $h\to 0$ in the  formula above, and using again that $\F$ is strongly lower semicontinuous, by the arbitrariness of $y$\,,  we get \eqref{eqpc}.

\vskip5pt
Finally, we prove that \eqref{dopo} holds true. Let $\eta^\tau:[0,+\infty)\to (0,\tau]$ be the function defined by $\eta^\tau(t)=(k+1)\tau-t$ for every $t\in[k\tau,(k+1)\tau)$\,.
 By \eqref{proprietà}, 
\begin{equation*}
\dot{x}^\tau(t)\in-\partial\F(x^{\tau}(t+\eta^\tau(t)))\qquad\textrm{ for every }t>0\,,
\end{equation*}
which, using \eqref{probcauchy} and \eqref{monsotdiff}, yields 
\begin{eqnarray*}
\frac{\ud}{\ud t}|x(t)-x^\tau(t)|^2_\h&=&2\langle x(t)-x^\tau(t),\dot{x}(t)-\dot{x}^\tau(t) \rangle_\h\\
&=&2\langle x(t)-x^\tau(t+\eta^\tau(t)),\dot{x}(t)-\dot{x}^\tau(t) \rangle_\h\\
&&+2\langle x^\tau(t+\eta^\tau(t))-x^\tau(t),\dot{x}(t)-\dot{x}^\tau(t) \rangle_\h\\
&\le& 4  \lambda|x(t)-x^\tau(t+\eta^\tau(t))|^2_\h+2 |x^\tau(t+\eta^\tau(t))-x^\tau(t)|_\h|\dot{x}(t)-\dot{x}^\tau(t)|_\h\\
&\le& 8 \lambda|x(t)-x^\tau(t)|^2_\h+ 8 \lambda |x^\tau(t+\eta^\tau(t))-x^\tau(t)|^2_\h\\
&&+2|x^\tau(t+\eta^\tau(t))-x^\tau(t)|_\h(|\dot{x}(t)|_\h+|\dot{x}^\tau(t)|_\h)\\
&\le& 8 \lambda|x(t)-x^\tau(t)|^2_\h+ \tau(8\lambda\tau+3)  (|\dot{x}^\tau(t)|^2_\h+|\dot{x}(t)|_\h^2)\,,\\
&\le& 8 \lambda|x(t)-x^\tau(t)|^2_\h+ 4 \tau (|\dot{x}^\tau(t)|^2_\h+|\dot{x}(t)|_\h^2)\,,
\end{eqnarray*}
where in the last inequality we have used that $\tau\le\frac{1}{16\lambda}$  (recall also 
 $|x^\tau(t+\eta^\tau(t))-x^\tau(t)|_\h=\eta^\tau(t)|\dot{x}^\tau(t)|_\h\leq \tau |\dot{x}^\tau(t)|_\h$) .
By integrating the equation above and using  \eqref{dopo0} and \eqref{bound100}, we get
\begin{equation*}
\begin{aligned}
|x(t)-x^\tau(t)|^2_\h\le& \int_{0}^t 8\lambda|x(s)-x^\tau(s)|^2_\h\ud s+4\tau\int_{0}^t(|\dot{x}^\tau(s)|^2_\h+|\dot{x}(s)|_\h^2)\ud s\\
\le&\int_{0}^t 8\lambda|x(s)-x^\tau(s)|^2_\h\ud s+8\tau 4^{8\lambda t+4}(|\F(x_0)|+\lambda|x_0|_\h^2)\,.
\end{aligned}
\end{equation*}
Setting $\alpha(s):=8\tau 4^{8\lambda s+4}(|\F(x_0)|+\lambda|x_0|_\h^2)$ and $\beta(s):=8\lambda$ for every $s>0$\,, we have
\begin{equation*}
\begin{aligned}
|x(t)-x^\tau(t)|^2_\h\le& \int_{0}^t \beta(s)|x(s)-x^\tau(s)|^2_\h\ud s+\alpha(t)\,.
\end{aligned}
\end{equation*}
Noticing that $\alpha$ is non-decreasing, by Gronwall Lemma we get
\begin{equation*}
\begin{aligned}
|x(t)-x^\tau(t)|^2_\h\le& \alpha(t)e^{\int_{0}^t\beta(s)\ud s}=8\tau 4^{8\lambda t+4}(|\F(x_0)|+\lambda|x_0|_\h^2)e^{8\lambda t}\,,
\end{aligned}
\end{equation*}
thus providing \eqref{dopo}.
\end{proof}
For every vectorial space $\mathscr{V}$  we denote by $\mathscr{V}^*$ the algebraic dual space of $\mathscr{V}$ and by $\mathscr{V}'$ the topological dual space of $\mathscr{V}$. 
\begin{proposition}\label{singleton}
Let $\F:\h\to(-\infty,+\infty]$ be a proper lower semicontinuous function which is $\lambda$-convex, for some $\lambda>0$ and let $x\in D(\F)$\,.
Let $\hat\h$ be a dense subspace of $\h$\,.
If there exists $T\in(\hat\h)^*$ such that 
\begin{equation}\label{grad}
\lim_{t\to 0}\frac{\F(x+t\ffi)-\F(x)}{t}=T(\ffi)\qquad\textrm{for every }\ffi\in\hat\h\,,
\end{equation}
then, either $\partial\F(x)=\emptyset$ or $\partial\F(x)=\{v\}$\,, where $v$ is the (unique) element in $\h$  satisfying $T(\ffi)=\langle v,\ffi\rangle_\h$ for every $\ffi\in\hat\h$\,. In particular,
$T\in (\hat\h)'$ and $v$ is its  unique continuous extension to $\h'$.
\end{proposition}
\begin{proof}
Since $\hat\h$ is dense in $\h$, in order to get the claim it is enough to prove that for every $v\in\partial\F(x)$
\begin{equation}\label{newcl}
\langle v,\ffi\rangle_\h=T(\ffi)\qquad\textrm{for every }\ffi\in\hat\h\,.
\end{equation}
To this purpose, we notice that every $v\in\partial\F(x)$ satisfies
\begin{equation*}
\F(x+t\ffi)-\F(x)-t\langle v,\ffi\rangle_\h\ge -t^2\frac{\lambda}{2}|\ffi|^2_\h\qquad\textrm{for every }\ffi\in\hat\h,\,t\in\R\,, 
\end{equation*}
which, dividing by $t$\,, yields
$$
\lim_{t\to 0}\frac{\F(x+t\ffi)-\F(x)}{t}=\langle v,\ffi\rangle_\h \qquad\textrm{for every }\ffi\in\hat\h\,;
$$
therefore, in view of \eqref{grad}, we get \eqref{newcl}.
\end{proof}
The following Theorem provides a  convergence result for  gradient-flows  associated to a  $\Gamma$-converging sequence   of functions satisfying the assumptions of Theorem~\ref{exist}.
\begin{theorem}\label{genstab}
Let $\{\F^n\}_{n\in\N}$ with $\F^n:\h\to(-\infty,+\infty]$ for every $n\in\N$ be a sequence of proper, strongly lower semicontinuous functions which are $\lambda$-convex and $\lambda$-positive, for some $\lambda>0$ independent of $n$\,. Let $\{x_0^n\}_{n\in\N}\subset\h$ be such that $x_0^n\in D(\F^n)$ for every $n\in\N$\,, $S:=\sup_{n\in\N}\F^n(x_0^n)<+\infty$ and $x_0^n\to x^\infty_0$
for some $x^\infty_0\in\h$\,. Assume that one of the following statements is satisfied:
\begin{itemize}
\item[(a)] The functions $\F^n$ are equicoercive and  $\Gamma$-converge to some proper function $\F^\infty$   with respect to the weak $\h$-convergence (as $n\to +\infty$)\,. Moreover, the $\Gamma$-limsup inequality is satisfied with respect to the strong $\h$-convergence, i.e., for every $y\in\h$ there exists a sequence $\{y^n\}_{n\in\N}$ with $y^n\overset{\h}{\to} y$ such that $\F^n(y^n)\to \F^{\infty}(y)$ as $n\to +\infty$\,.
\item[(b)] The functions $\F^n$ $\Gamma$-converge to some proper function $\F^\infty$ with respect to the strong $\h$-convergence (as $n\to +\infty$) and every sequence $\{y^n\}_{n\in\N}\subset\h$ with $\sup_{n\in\N}\F^n(y^n)+\frac{\lambda}{2}|y^n|_\h^2<+\infty$\,, admits a strongly convergent subsequence.
\end{itemize}
Then, $x_0^\infty\in D(\F^{\infty})$ and, for every $T>0$\,, the solutions $x^n$ to the Cauchy problem
\begin{equation}\label{caun}
\begin{cases}
 \dot{x}(t)\in -\partial \F^n(x(t))  \qquad \text{for a.e.\ }t \in (0,T)\,, \\
x(0)=x_0^n
\end{cases}
\end{equation}  
weakly converge, as $n\to+\infty$\,, in $H^1([0,T];\h)$ to the unique solution $x^\infty$ to the problem
\begin{equation}\label{cauninf}
\begin{cases}
 \dot{x}(t)\in -\partial \F^\infty(x(t)) \qquad \text{for a.e.\ }t \in (0,T)\,, \\
x(0)=x^\infty_0\,.
\end{cases}
\end{equation}  
Furthermore, if 
\begin{equation}\label{reco0}
\lim_{n\to +\infty} \F^n(x^n_0)=\F^\infty(x^\infty_0)\,,
\end{equation} 
then,  we have that 
\begin{equation}\label{strongc}
x^n\to x^\infty\qquad\textrm{(strongly) in }H^1([0,T];\h)\quad\textrm{ for every }T>0\,,
\end{equation}
\begin{equation}\label{alwaysrs}
x^n(t)\overset{\h}{\to}x^{\infty}(t)\quad\textrm{and}\quad\F^n(x^n(t))\to\F^{\infty}(x^{\infty}(t))\qquad\textrm{for every }t\ge 0\,. 
\end{equation}
\end{theorem}
\begin{proof}
We preliminarily notice that, if either (a) or (b) is satisfied, then the function $\F^\infty$ is strongly lower semicontinuous, $\lambda$-convex and $\lambda$-positive and $x^\infty_0\in D(\F^\infty)$\,.
 Moreover, by Theorem~\ref{exist}, for every $n \in \N$ there exists a unique solution $x^n$ to \eqref{caun}. 

Let $0<\tau<\frac{1}{2\lambda}$ and let $\{x^{\infty,\tau}_k\}_{k\in\N}$ denote the discrete-in-time evolution in \eqref{movmin} for $x_0:=x_0^{\infty}$ and $\F:=\F^{\infty}$\,.
Analogously, for every $n\in\N$\,, let $\{x^{n,\tau}_k\}_{k\in\N}$ denote the discrete-in-time evolution in \eqref{movmin} for $x_0:=x_0^{n}$ and $\F:=\F^{n}$\,.
By Proposition \ref{exisitminprobincre}, $\{x^{\infty,\tau}_k\}_{k\in\N}$ and $\{x^{n,\tau}_k\}_{k\in\N}$ are uniquely determined.
Furthermore, for every $k\in\N$ we set
\begin{eqnarray*}
\mathcal{I}^{n,\tau}_k(\cdot)&:=&\F^n(\cdot)+\frac{1}{2\tau}|\cdot-x^{n,\tau}_{k-1}|_\h^2\qquad\textrm{for every }n\in\N\,,\\
\mathcal{I}^{\infty,\tau}_k(\cdot)&:=&\F^\infty(\cdot)+\frac{1}{2\tau}|\cdot-x^{\infty,\tau}_{k-1}|_\h^2\,.
\end{eqnarray*}

We first show that, if either (a) or (b) is satisfied, then for every $k\in\N$
\begin{equation}\label{convdisc}
\F^n(x^{n,\tau}_k)\to \F^\infty(x^{\infty,\tau}_k)\quad\textrm{and}\quad|x^{n,\tau}_k-x^{\infty,\tau}_k|_\h\to 0\qquad\textrm{as }n\to +\infty\,.
\end{equation}
 By finite induction, it is enough to show \eqref{convdisc} for $k=1$\,. We distinguish the two cases in which either (a) or (b) holds true. 

Assume first that (a) holds true. 
By the assumptions on $x^n_0$\,, we have that
$$
\mathcal{I}^{n,\tau}_1(x^{n,\tau}_1)\le \F^n(x^n_0)\le S\,,
$$
whence, using that for $\frac{1}{2\tau}>\lambda$ the functions $\mathcal{I}^{n,\tau}_1(\cdot)$ are  equicoercive, we deduce that, up to a subsequence, $x^{n,\tau}_1\overset{\h}{\weakly} y_1$ for some  $y_1\in\h$\,. 
Moreover, since $|x_0^n- x_0^\infty|_\h\to 0$ as $n\to +\infty$ and since the functions $\F^n$ $\Gamma$-converge to the function $\F^\infty$ with respect to the weak $\h$-convergence, we have that
\begin{equation}\label{i}
\mathcal{I}^{\infty,\tau}_1(y_1)\le\liminf_{n\to +\infty}\F^n(x^{n,\tau}_1)+\frac{1}{2\tau}\liminf_{n\to +\infty}|x^{n,\tau}_1-x^{n}_0|_\h^2\le\liminf_{n\to +\infty} \mathcal{I}^{n,\tau}_1(x^{n,\tau}_1)\,.
\end{equation}
Furthermore, since the $\Gamma$-limsup inequality is satisfied with respect to the strong $\h$-convergence, there exists $\{\bar{x}^{n,\tau}_1\}_{n\in\N}\subset\h$ such that 
\begin{equation}\label{ii}
\bar{x}^{n,\tau}_1\overset{\h}{\to} x^{\infty,\tau}_1\quad\textrm{ and }\quad\F^n(\bar{x}^{n,\tau}_1)\to\F(x^{\infty,\tau}_1)\,,
\end{equation} 
where $x^{\infty,\tau}_1$ is the unique solution to the problem \eqref{movmin} with $\F=\F^\infty$ and $k=1$\,.  Therefore, by \eqref{i} and \eqref{ii}, we get
\begin{equation*}
\begin{aligned}
\mathcal{I}^{\infty,\tau}_1(y_1)\le&\liminf_{n\to +\infty}\F^n(x^{n,\tau}_1)+\frac{1}{2\tau}\liminf_{n\to +\infty}|x^{n,\tau}_1-x_0^n|_\h^2\\
\le&\liminf_{n\to +\infty} \mathcal{I}^{n,\tau}_1(x^{n,\tau}_1)
\le\limsup_{n\to +\infty} \mathcal{I}^{n,\tau}_1(x^{n,\tau}_1)\\
\le&\lim_{n\to+\infty} \mathcal{I}^{n,\tau}_1(\bar{x}^{n,\tau}_1)=\mathcal{I}^{\infty,\tau}_1(x^{\infty,\tau}_1)\,,
\end{aligned}
\end{equation*}
whence, by the minimality of $x^{\infty,\tau}_1$\,, we deduce that all the inequalities above are in fact equalities and, in particular, that $y_1$ is a minimizer of $\mathcal{I}^{\infty,\tau}_1$;  in view of the uniqueness of the minimizer of $\mathcal{I}^{\infty,\tau}_1$, we deduce that $y_1=x^{\infty,\tau}_1$\,. 
By Urysohn Lemma, this implies that the whole sequence $\{x^{n,\tau}_1\}_{n\in\N}$ weakly converges to $x^{\infty,\tau}_1$\,.
Moreover, using that 
$$
\F^\infty(x^{\infty,\tau}_1)+\frac{1}{2\tau}|x^{\infty,\tau}_1-x^\infty_0|^2_\h=\liminf_{n\to +\infty}\F^n(x^{n,\tau}_1)+\frac{1}{2\tau}\liminf_{n\to +\infty}|x^{n,\tau}_1-x_0^n|_\h^2\,,
$$
since
$$
\F^\infty(x^{\infty,\tau}_1)\le\liminf_{n\to +\infty}\F^n(x^{n,\tau}_1)\quad\textrm{ and }\quad \frac{1}{2\tau}|x^{\infty,\tau}_1-x^\infty_0|^2_\h\le \frac{1}{2\tau}\liminf_{n\to +\infty}|x^{n,\tau}_1-x_0^n|_\h^2\,,
$$
we deduce that 
$$
\F^\infty(x^{\infty,\tau}_1)=\lim_{n\to +\infty}\F^n(x^{n,\tau}_1)\quad\textrm{ and }\quad|x^{\infty,\tau}_1-x^\infty_0|_\h=\lim_{n\to +\infty}|x^{n,\tau}_1-x^n_0|_\h\,,
$$
which 
implies \eqref{convdisc}  (for $k=1$ and then for all $k\in \N$). 

Assume now that (b) holds true.
 As above  (recall $\lambda< \frac{1}{2\tau}$)  we have that
$$
\F^n(x^{n,\tau}_1)+\lambda|x^{n,\tau}_1-x^{n}_0|^2_\h\le\mathcal{I}^{n,\tau}_1(x^{n,\tau}_1)\le S\,,
$$
whence, by the strong compactness property of the functions $\F^n(\cdot)+\frac{\lambda}{2}|\cdot|^2_\h$ we deduce that, up to a subsequence, $x^{n,\tau}_1\overset{\h}{\to} y_1$ for some  $y_1\in\h$\,. 
Moreover, since $|x_0^n- x_0^\infty|_\h\to 0$ as $n\to +\infty$\,, we have that the  functionals $\mathcal{I}^{n,\tau}_1$  $\Gamma$-converge with respect to the strong-$\h$ convergence to the  functional  $\mathcal{I}^{\infty,\tau}_1$\,.
By the fundamental theorem of $\Gamma$-convergence and by the uniqueness of the minimizer of the problem \eqref{movmin} with $\F=\F^\infty$ and $k=1$\,, we get that $y_1=x^{\infty,\tau}_1$\,, that the whole sequence $\{x^{n,\tau}_1\}_{n\in\N}$ strongly converges to $x^{\infty,\tau}_1$\,, and that \eqref{convdisc} is satisfied for $k=1$\,. This concludes the proof of \eqref{convdisc} for both the cases (a) and (b).

\vskip5pt
Now we show that for every $T>0$\,, $x^n\weakly x^\infty$ in $H^1([0,T];\h)$\,, where $x^\infty$ is the unique solution to \eqref{cauninf}\,.
To this end, we first notice that, for $n$ large enough, 
$|\F^n(x_0^n)|\le S + 2 \lambda |x_0|^2$, so that $|\F^n(x_0^n)| + \lambda |x_0^n|^2 \le S + 4 \lambda |x_0|^2$; therefore, 
 by \eqref{dopo0}
$$
\|\dot{x}^n\|^2_{L^2((0,T);\h)}\le  4^{8\lambda T+4}(S+4\lambda|x_0|_\h^2)\,,
$$
so that, up to a subsequence, $x^n\weakly \bar x$ in  $H^1([0,T];\h)$\,, for some $\bar x\in H^1([0,T];\h)$\,.
Now we show that $\bar x=x^\infty$\,. 

For every $0<\tau<\frac{1}{2\lambda}$\,, let $x^{\infty,\tau}$ and $x^{n,\tau}$ ($n\in\N$) denote the piecewise affine interpolations defined in \eqref{funzdeftratti}, of $\{x^{\infty,\tau}_k\}_{k\in\N}$ and $\{x^{n,\tau}_k\}_{k\in\N}$, respectively.
By \eqref{convdisc}, we have that
\begin{equation}\label{convdisc2}
\lim_{n\to +\infty}|x^{n,\tau}(t)-x^{\infty,\tau}(t)|_\h=0\qquad\textrm{for every }t>0\,,0<\tau<\frac{1}{2\lambda}\,.
\end{equation}
Let $t>0$\,. For every $0<\tau<\frac{1}{2\lambda}$\,, by triangle inequality and by \eqref{dopo} we have that
\begin{equation}\label{trine}
\begin{aligned}
|x^{n}(t)-x^{\infty}(t)|_\h\le& |x^{n}(t)-x^{n,\tau}(t)|_\h+|x^{n,\tau}(t)-x^{\infty,\tau}(t)|_\h+|x^{\infty,\tau}(t)-x^{\infty}(t)|_\h\\
\le&
 16  \tau 12^{8\lambda t +4 }\big(S+4\lambda|x_0|_\h^2\big)
+|x^{n,\tau}(t)-x^{\infty,\tau}(t)|_\h\,;
\end{aligned}
\end{equation}
therefore, sending first $n\to +\infty$ and then $\tau\to 0$ in \eqref{trine} and using \eqref{convdisc2}, we get that $x^n(t)\overset{\h}{\to} x^{\infty}(t)$ as $n\to +\infty$\,.
By the uniqueness of the limit we deduce that $\bar x=x^{\infty}$ and that the whole sequence $\{x^n\}_{n\in\N}$ weakly  converges  in $H^1([0,T];\h)$ to $x^\infty$\,.
\vskip5pt
Finally, we prove that \eqref{strongc} and \eqref{alwaysrs} hold true. By \eqref{trine}, the first part of \eqref{alwaysrs} is satisfied. Moreover, by \cite[formula (1.10)]{RS}  (notice that, as observed in \cite{RS}, the formula applies also for $\lambda$-convex energies),  we have that, for every $t>0$\,,
\begin{equation}\label{enid} 
\begin{aligned}
&\F^n(x^n_0(t))-\F^n(x^n(t))=\frac 1 2\int_{0}^{t}|\dot{x}^n(s)|^2_\h\ud s\qquad\textrm{for every }n\in\N\,,\\
&\F^\infty(x^\infty_0(t))-\F^\infty(x^\infty(t))=\frac 1 2\int_{0}^{t}|\dot{x}^\infty(s)|^2_\h\ud s\,,
\end{aligned}
\end{equation}
which, using \eqref{reco0}, the $\Gamma$-liminf inequality (that holds true in both the cases (a) and (b)) and the weak $H^1$-convergence of $x^n$ to $x^\infty$, implies
\begin{equation*}
\begin{aligned}
&\F^\infty(x^\infty_0(t))-\liminf_{n\to +\infty}\F^n(x^n(t))\le
\F^\infty(x^\infty_0(t))-\F^\infty(x^\infty(t))
=\frac 1 2\int_{0}^{t}|\dot{x}^\infty(s)|^2_\h\ud s\\
\le& \liminf_{n\to +\infty}\frac 1 2\int_{0}^{t}|\dot{x}^n(s)|^2_\h\ud s
\le  \limsup_{n\to +\infty}\frac 1 2\int_{0}^{t}|\dot{x}^n(s)|^2_\h\ud s\\
\le&\limsup_{n\to +\infty} \F^n(x^n_0(t))-\liminf_{n\to +\infty}\F^n(x^n(t))
=\F^\infty(x^\infty_0(t))-\liminf_{n\to +\infty}\F^n(x^n(t))\,.
\end{aligned}
\end{equation*}
Therefore, all the inequalities above are actually equalities;
in particular,
\begin{equation}\label{convnorme}
\frac 1 2\int_{0}^{t}|\dot{x}^\infty(s)|^2_\h\ud s=\lim_{n\to +\infty}\frac 1 2\int_{0}^{t}|\dot{x}^n(s)|^2_\h\ud s\,,
\end{equation}
which, together \eqref{reco0} and \eqref{enid}, yields
$$
\F^\infty(x^\infty(t))=\lim_{n\to +\infty}\F^n(x^n(t))\,,
$$
thus obtaining also the second part of \eqref{alwaysrs}. 
Finally, by \eqref{convnorme}, we obtain also \eqref{strongc}, thus concluding the proof of the theorem.
\end{proof}

\section{Convergence of the $s$-fractional heat flows} \label{Sec4}
 This section is devoted to the proof of the stability of the $s$-fractional heat flows as $s\to 0^+$ and $s\to 1^-$.
In the first part, we define the $s$-fractional laplacian for $s\in (0,1)$ and for $s=0$\,. The second part contains the convergence theorems, which are the main results of the paper. 
 In this section we denote by $v_t$ the partial time derivative of a function $v$\,.
\subsection{The $s$-fractional laplacian for $s\in (0,1)$ and for $s=0$}
For every $s\in(0,1)$ and for every $\psi\in \CC_\cc^\infty(\R^d)$  the $s$-fractional laplacian of $\psi$ is defined by
\begin{equation*}
(-\Delta)^s\psi(x):=\int_{\R^d}\frac{2\psi(x)-\psi(x+z)-\psi(x-z)}{|z|^{d+2s}}\ud z,\qquad x\in\R^d\,.
\end{equation*}
In \cite[Lemma 3.2]{DPV} it is proven that the above integral is finite, that $(-\Delta)^s\psi\in L^\infty(\R^d)$\,, and that 
\begin{equation}\label{PV}
(-\Delta)^s\psi(x)=2\lim_{r\to 0^+}\int_{\R^d\setminus B_r(0)}\frac{\psi(x)-\psi(x+z)}{|z|^{d+2s}}\ud z\,.
\end{equation}
For every $u\in\mathcal{H}^s_0(\Omega)$ we define the $s$-fractional laplacian of $u$ by duality as
\begin{equation}\label{extlap}
\langle (-\Delta)^su,\ffi\rangle:=\langle u,(-\Delta)^s\tilde\ffi\rangle,\qquad\textrm{for all }\ffi\in \CC_{\cc}^{\infty}(\Omega). 
\end{equation}
Here and below $\langle\cdot,\cdot\rangle$ denotes the standard scalar product in $L^2$.
Clearly, the $s$-fractional laplacian is nothing but the first variation of the squared Gagliardo $s$-norm, as shown below.
\begin{proposition}\label{lemma:firstvar}
Let $s\in (0,1)$\,. For every $u\in \mathcal{H}^s_0(\Omega)$ and for every $\ffi\in \CC_\cc^\infty(\Omega)$ we have
\begin{equation}\label{eq:firstvar}
\lim_{t\to 0}\frac{F^s(u+t\ffi)-F^s(u)}{t}=\langle(-\Delta)^su,\ffi\rangle.
\end{equation}
\end{proposition}
\begin{proof}
We have
\begin{eqnarray*}
\lim_{t\to 0}\frac{F^s(u+t\ffi)-F^s(u)}{t}&=&\int_{\R^d}\int_{\R^d}\frac{(\tilde u(x)-\tilde u(y))(\tilde\ffi(x)-\tilde\ffi(y))}{|x-y|^{d+2s}}\ud y\ud x\\
&=&\int_{\Omega} u(x)\lim_{r\to 0^+}\int_{\R^d\setminus B_r(0)}\frac{\tilde\ffi(x)-\tilde\ffi(x+z)}{|z|^{d+2s}}\ud z\ud x\\
&&+\int_{\Omega} u(y)\lim_{r\to 0^+}\int_{\R^d\setminus B_r(0)}\frac{\tilde\ffi(y)-\tilde\ffi(y-z)}{|z|^{d+2s}}\ud z\ud y\\
&=&\langle u,(-\Delta)^s\tilde\ffi\rangle=\langle (-\Delta)^su,\ffi\rangle,
\end{eqnarray*}
where we have used the change of variable $z=y-x$\,, \eqref{PV} and \eqref{extlap}.
\end{proof}
For every $\psi\in\CC^{\infty}_{\cc}(\R^d)$ we define the $0$-fractional laplacian of $\psi$ as
\begin{equation*}
(-\Delta)^0\psi(x):=\int_{B_1}\frac{2\psi(x)-\psi(x+z)-\psi(x-z)}{|z|^d}\ud z-2\int_{\R^d\setminus \overline{B}_1}\frac{\psi(x+z)}{|z|^d}\ud z\,,\quad x\in\R^d\,.
\end{equation*}
We notice that $(-\Delta)^0\psi$ is well-defined for every $\psi\in\CC^{\infty}_{\cc}(\R^d)$ since
\begin{equation*}
\int_{B_1}\frac{|2\psi(x)-\psi(x+z)-\psi(x-z)|}{|z|^d}\ud z\le 2\int_{B_1}\frac{|\psi(x+z)-\psi(x)|}{|z|^d}\ud z\\
\le C [\psi]_{0,1}\,.
\end{equation*} 
and 
\begin{equation*}
\begin{aligned}
\int_{\R^d\setminus \overline{B}_1}\frac{|\psi(x+z)|}{|z|^d}\ud z\le \|\psi\|_{L^1}\,.
\end{aligned}
\end{equation*}
\begin{remark}\label{ChenWeth}
\rm
In \cite{CW} the notion of {\it logarithmic laplacian} $L_{\Delta}$ has been introduced as follows
$$
\mathrm{L}_\Delta\psi(x):=c_{d,1}(-\Delta)^0\psi(x)+c_{d,2}\psi(x)\,,
$$
where $c_{d,1}$ and $c_{d,2}$ are specific constant depending only on the dimension $d$\,.
Such a logarithmic laplacian would correspond to renormalizing the Gagliardo $s$-seminorm of $\psi$ by removing all but a finite amount of the blowing up quantity $\frac{\|\psi\|^2_{L^2}}{s}$\,.
\end{remark}
For every $u\in\mathcal{H}^0_0(\Omega)$
we define $0$-fractional laplacian of $u$ by duality as 
\begin{equation}\label{ext0lap}
\langle (-\Delta)^0u,\ffi\rangle :=\langle u,(-\Delta)^0\tilde\ffi\rangle\,,\qquad\textrm{for all }\ffi\in\CC^{\infty}_{\cc}(\Omega)\,.
\end{equation}
Clearly, the $0$-fractional laplacian is the first variation of the functional $\hat F^0$\,, as shown in the following result (we recall that the functions $G^0_1$ and $J^0_1$ have been introduced in \eqref{HsR} and \eqref{Jsr}, respectively).
\begin{proposition}\label{lemma:firstvar0}
For every $u\in \mathcal{H}^0_0(\Omega)$ and for every $\ffi\in \CC^{\infty}_{\cc}(\Omega)$ we have
\begin{eqnarray}\label{eq1:firstvar0}
\lim_{t\to 0}\frac{G^0_1(u+t\ffi)-G^0_1(u)}{t}&=&\Big\langle u,\int_{B_1}\frac{2\tilde\ffi(x)-\tilde\ffi(x+z)-\tilde\ffi(x-z)}{|z|^d}\ud z\Big\rangle\\ \label{eq2:firstvar0}
\lim_{t\to 0}\frac{J^0_1(u+t\ffi)-J^0_1(u)}{t}&=&\Big\langle u, -2\int_{\R^d\setminus \overline{B}_1}\frac{\tilde\ffi(x+z)}{|z|^d}\ud z\Big\rangle\,,
\end{eqnarray}
so that
\begin{equation}\label{eq3:firstvar0}
\lim_{t\to 0}\frac{\hat F^0(u+t\ffi)-\hat F^0(u)}{t}=\langle (-\Delta)^0u,\ffi\rangle\,.
\end{equation}
\end{proposition}
\begin{proof}
Fix $u\in \mathcal{H}^0_0(\Omega)$ and $\ffi\in\CC_{\cc}^{\infty}(\Omega)$\,.
Then, using the change of variable $z=y-x$\,, we have
\begin{equation*}
\begin{aligned}
&\lim_{t\to 0}\frac{G^0_1(u+t\ffi)-G^0_1(u)}{t}=\iint_{\B_1}\frac{(\tilde u(x)-\tilde u(y))(\tilde\ffi(x)-\tilde\ffi(y))}{|x-y|^d}\ud y\ud x\\
=&\int_{\Omega}u(x)\int_{B_1}\frac{\tilde\ffi(x)-\tilde\ffi(x+z)}{|z|^d}\ud z\ud x+\int_{\Omega}u(y)\int_{B_1}\frac{\tilde\ffi(y)-\tilde\ffi(y-z)}{|z|^d}\ud z\ud y\\
=&\Big\langle u,\int_{B_1}\frac{2\tilde\ffi(x)-\tilde\ffi(x+z)-\tilde\ffi(x-z)}{|z|^d}\ud z \Big\rangle\,, 
\end{aligned}
\end{equation*}
i.e., \eqref{eq1:firstvar0}.
Moreover, using again the change of variable $z=y-x$, we obtain
\begin{equation*}
\begin{aligned}
&\lim_{t\to 0}\frac{J^0_1(u+t\ffi)-J^0_1(u)}{t}=-2\iint_{\R^{2d}\setminus\overline{\B}_1}\frac{\tilde u(x)\tilde\ffi(y)}{|x-y|^d}\ud y\ud x\\
=&\int_{\Omega} u(x)\Big(-2\int_{\R^d\setminus \overline{B}_1(x)}\frac{\tilde\ffi(y)}{|x-y|^d}\ud y\Big)\ud x\,,\\
=&\Big\langle u, -2\int_{\R^d\setminus \overline{B}_1}\frac{\tilde\ffi(x+z)}{|z|^d}\ud z\Big\rangle\,,
\end{aligned}
\end{equation*}
namely, \eqref{eq2:firstvar0}.
Finally, \eqref{eq3:firstvar0} follows from \eqref{eq1:firstvar0} and \eqref{eq2:firstvar0}, using \eqref{ext0lap}.
\end{proof}
\subsection{The main results}
Here we first state and then prove the  convergence results for the parabolic flows corresponding to the (either scaled or renormalized) $s$-Gagliardo seminorms as $s\to 0$ and $s\to 1$.  These follow by collecting the preparatory results of the previous sections, and Lemma~\ref{lambdapc} for the first order convergence as $s\to 0^+$. 
 
 We start with the convergences as $s\to 0^+$.
\begin{theorem}\label{convheat0ordsto0}
Let $\{s_n\}_{n\in\N}\subset (0,1)$ be such that $s_n\to 0^+$ as $n\to +\infty$\,. Let $u^0_0\in L^2(\Omega)$ and let $\{u_0^n\}_{n\in\N}\subset L^2(\Omega)$ be such that $u^n_0\in\mathcal{H}^{s_n}_0(\Omega)$\,, $S:=\sup_{n\in\N}s_nF^{s_n}(u^n_0)<+\infty$ and 
 $u^n_0\to u^0_0$ in $L^2(\Omega)$\,.
Then,
for every $n\in\N$ there exists a unique  solution $u^n\in H^1([0,+\infty); L^2(\Omega))$ to
\begin{equation}\label{cauchyord0n}
\begin{dcases}
 u_t(t) =-s_n(-\Delta)^{s_n}u(t) \qquad\textrm{for a.e.\ }t\in[0,+\infty)\\
u(0)=u^{n}_0\,,
\end{dcases}
\end{equation}
 and such a solution satisfies   $(-\Delta)^{s_n}u^n(t)\in L^2(\Omega)$ for  $a.e. \, t\ge 0$\,.
Moreover, for every $T>0$, $u^n\weakly u^0$ in $H^1([0,T];L^2(\Omega))$ as $n\to +\infty$\,, where $u^0\in H^1([0,T];L^2(\Omega))$ is the unique 
solution to
\begin{equation}\label{cauchyord0infty}
\begin{dcases}
 u_t(t) =-d\omega_d u(t)\qquad\textrm{for a.e.\ }t\in (0,T)\,,\\
u(0)=u^0_0\,.
\end{dcases}
\end{equation}
Furthermore, if
$$
\lim_{n\to +\infty} s_nF^{s_n}(u^n_0)=F^0(u^0_0)\,,
$$ 
then, $u^n\to u^0$ (strongly) in $H^1([0,T];L^2(\Omega))$ for  every  $T>0$\,, and
$$
\|u^n(t)-u^0(t)\|_{L^2(\Omega)}\to 0\quad\textrm{and}\quad s_nF^{s_n}(u^n(t))\to F^0(u^0(t))\qquad\textrm{for every }t\ge 0\,.
$$
\end{theorem}
\begin{theorem}\label{convheat1ordsto0}
Let $\{s_n\}_{n\in\N}\subset (0,1)$ be such that $s_n\to 0^+$ as $n\to +\infty$\,. Let $u^0_0\in L^2(\Omega)$ and let $\{u^{n}_0\}_{n\in\N}\subset L^2(\Omega)$ be such that $u^n_0\in\mathcal{H}^{s_n}_0(\Omega)$\,, $S:=\sup_{n\in\N}\hat{F}^{s_n}(u^n_0)<+\infty$
and
 $u^n_0\to u^0_0$ in $L^2(\Omega)$\,. Then, 
for every $n\in\N$ there exists a unique  solution $u^n\in H^1([0,+\infty); \mathcal{H}^{s_n}_0(\Omega))$ to
\begin{equation}\label{cauchyord1n}
\begin{dcases}
 u_t(t) =- \Big[(-\Delta)^{s_n}u(t)-\frac{d\omega_d}{s_n}u(t) \Big] \qquad\textrm{for a.e.\ }t\in (0,T)\,,\\
u(0)=u^{n}_0\,,
\end{dcases}
\end{equation}
 and such a solution satisfies  $(-\Delta)^{s_n}u^n(t)\in L^2(\Omega)$ for  $a.e.\, t\ge 0$\,.
Moreover, $u^0_0\in\mathcal{H}^0_0(\Omega)$ and, for every $T>0$,  $u^n\to u^0$ in $H^1([0,T];L^2(\Omega))$ as $n\to +\infty$\,, where $u^0\in H^1([0,T];\mathcal{H}^0_0(\Omega))$ is the unique solution to
\begin{equation}\label{cauchyord1infty}
\begin{dcases}
 u_t(t) =-(-\Delta)^0 u(t)\qquad\textrm{for a.e.\ }t\in (0,T)\\
u(0)=u^0_0\,,
\end{dcases}
\end{equation}
 and such a solution satisfies  $(-\Delta)^0 u^0(t)\in L^2(\Omega)$ for $a.e. \,t\ge 0$\,.
Furthermore, if
$$
\lim_{n\to +\infty} \hat{F}^{s_n}(u^n_0)=\hat F^0(u^0_0)\,,
$$
then,  $u^n\to u^0$ (strongly) in $H^1([0,T];L^2(\Omega))$ for  every  $T>0$\,, and
$$
\|u^n(t)-u^0(t)\|_{L^2(\Omega)}\to 0\quad\textrm{and}\quad \hat{F}^{s_n}(u^n(t))\to \hat{F}^0(u^0(t))\qquad\textrm{for every }t\ge 0\,.
$$
\end{theorem}
 The result below shows the convergence toward the classical heat equation as $s\to 1^-$ of the rescaled in time $s$-fractional heat equations. 
\begin{theorem}\label{11genstab}
Let $\{s_n\}_{n\in\N}\subset (0,1)$ be such that $s_n\to 1^-$ as $n\to +\infty$\,. Let $u^\infty_0\in L^2(\Omega)$ and let $\{u_0^n\}_{n\in\N}\subset L^2(\Omega)$ be such that $u^n_0\in\mathcal{H}^{s_n}_0(\Omega)$\,, $S:=\sup_{n\in\N}(1-s_n)F^{s_n}(u^n_0)<+\infty$ and 
 $u^n_0\to u^\infty_0$ in $L^2(\Omega)$\,.
Then,
for every $n\in\N$ there exists a unique  solution 
$u^n\in H^1([0,+\infty);  L^2(\Omega) )$ to
\begin{equation}\label{11caun}
\begin{dcases}
 u_t(t) =-(1-s_n)(-\Delta)^{s_n}u(t)  \qquad\textrm{for a.e.\ }t\in[0,+\infty) \\
u(0)=u^{n}_0\,,
\end{dcases}
\end{equation}
  and such a solution satisfies  $(-\Delta)^{s_n}u^n(t)\in L^2(\Omega)$ for  $a.e.\, t\ge 0$\,.
Moreover, $u^\infty_0\in H^1_0(\Omega)$\,, and, for every $T>0$, $u^n\weakly u^\infty$ in $H^1([0,T];L^2(\Omega))$ as $n\to +\infty$\,, where $u^\infty\in H^1([0,T];H^1_0(\Omega))$ is the unique  solution to
\begin{equation}\label{11cauninf}
\begin{dcases}
 u_t(t) =\frac{\omega_d}{2} \Delta u(t)  \qquad\textrm{for a.e.\ }t\in[0,+\infty)\\
u(0)=u^\infty_0\,.
\end{dcases}
\end{equation}
Furthermore, if
\begin{equation*}
\lim_{n\to +\infty} (1-s_n)F^{s_n}(u^n_0)=F^1(u^\infty_0)\,,
\end{equation*} 
then, $u^n\to u^\infty$ (strongly) in $H^1([0,T];L^2(\Omega))$ for  every  $T>0$\,, and
$$
\|u^n(t)-u^\infty(t)\|_{L^2(\Omega)}\to 0\quad\textrm{and}\quad (1-s_n)F^{s_n}(u^n(t))\to F^1(u^\infty(t))\qquad\textrm{for every }t\ge 0\,.
$$
\end{theorem}
 
We first prove Theorem \ref{convheat0ordsto0}.
 \begin{proof}[Proof of Theorem \ref{convheat0ordsto0}]
 By the very definition of $F^{s}$ in \eqref{defFs}, we have that for every  $n\in\N$ 
 $D(s_nF^{s_n}) =\mathcal{H}^{s_n}_0(\Omega)\neq\emptyset$ and that the functionals $s_nF^{s_n}$ are strongly lower semicontinuous, $\lambda$-positive and $\lambda$-convex  for every $\lambda>0$\,. Moreover, by  combining  Proposition~\ref{lemma:firstvar} with Proposition~\ref{singleton}  for  $\F=s_nF^{s_n}$\,, $\h=L^2(\Omega)$, and $\hat{\h}=\CC^\infty_\cc(\Omega)$, we have that  for every $u\in\mathcal{H}^{s_n}_0(\Omega)$\,, either $\partial (s_nF^{s_n})(u)=\emptyset$ or $\partial (s_nF^{s_n})(u)=\{(-\Delta)^{s_n}u\}$ with $s_n(-\Delta)^{s_n}u\in L^2(\Omega)$\,. 
Therefore, by Theorem~\ref{exist}, there exists a unique solution to the Cauchy problem \eqref{cauchyord0n}, with $s_n(-\Delta)^{s_n}u\in L^2(\Omega)$ for $a.e. \, t\ge 0$, for every $n\in\N$.  Furthermore, for every $u\in L^2(\Omega)$ we have that
\begin{equation}\label{fvf0}
\lim_{t\to 0}\frac{F^0(u+t\ffi)-F^0(u)}{t}=d\omega_d\langle u,\ffi\rangle_{L^2(\Omega)}\qquad\textrm{for every }\ffi\in L^2(\Omega)\,,
\end{equation}
whence we deduce that $\partial F^0(u)=\{d\omega_d\, u\}$\,. As a consequence, there exists a unique solution to the problem \eqref{cauchyord0infty}.
 Finally, the stability claims follow by applying Theorem \ref{genstab} with $\F^n=s_nF^{s_n}$ and $\F^\infty=F^0$\,, once noticed that, in view of Theorem \ref{gammafs}, assumption (a) is satisfied.
 \end{proof}
  In order to prove Theorem~\ref{convheat1ordsto0}, we provide below a lemma showing uniform $\lambda$-convexity of the underlying functionals. 
\begin{lemma}\label{lambdapc}
	For every $\lambda>2|\Omega|$\,, the functionals $\hat{F}^s$ are $\lambda$-positive and $\lambda$-convex for every $s\in[0,1)$\,.
\end{lemma}
\begin{proof}
As for the $\lambda$-positivity it is enough to notice that, by the very definition of $\hat F^s$ in \eqref{reprs} and \eqref{repr0} and by \eqref{estJsr}, recalling that $G^s_1\ge 0$ for every $s\in[0,1)$\,, we have that
	\begin{equation*}
	\hat{F}^s(u)+ \frac{\lambda}{2}\| u \|_{L^2(\Omega)}^2
	\geq G_1^s(u)+ \left (\frac{\lambda}{2}-\vert \Omega \vert \right ) \|u \|_{L^2(\Omega)}^2\geq 0\,.
	\end{equation*}
	Now we show that the functionals $\hat F^s$ are $\lambda$-convex for every $s\in[0,1)$\,.
	We preliminarily notice that the functionals $G^s_1$ are convex for every $s\in[0,1)$\,. Therefore, it is enough to show that the functionals $J^s_1$ are $\lambda$-convex. To this end, for every 	 $ u, v\in \mathcal{H}_0^0(\Omega)$ we define the function
	$$
	f: \R \rightarrow \R,\qquad f(t):=  J_1^s(u+t v)+\frac{\lambda}{2}\|u+t v \|_{L^2(\Omega)}^2
	$$
	and	we claim that $\frac{\ud^2}{\ud t^2}f(t)\geq 0$ for every $t \in \R$\,.
	Indeed, since
	\begin{equation*}
	J_1^s(u+t v)= J_1^s(u)-2t \iint_{\R^{2d}\setminus\overline{\mathcal{B}}_1}\frac{\tilde{u}(x)\tilde{v}(y)}{\vert x-y \vert^{d+2s}}\ud x\ud y+t^2J_1^s(v)
	\end{equation*}
	and 
	\begin{equation*}
 \|u+ t v\|_{L^2(\Omega)}^2= \|u \|_{L^2(\Omega)}^2+2t\int_{\Omega} u(x)v(x)\ud x+ t^2\| v \|_{L^2(\Omega)}^2\,,
	\end{equation*}
	by \eqref{estJsr} we have
	$$\frac{\ud^2}{\ud t^2}f(t)= 2 J_1^s(v)+\lambda  \|  v  \|_{L^2(\Omega)}^2 \geq (-2 \vert \Omega \vert+ \lambda)\| v \|_{L^2(\Omega)}^2 \geq 0\,,  $$
	which implies the $\lambda$-convexity of the functional $J^s_1$ and then the $\lambda$-convexity of $\hat F^s$\,. 
\end{proof}
\begin{proof}[Proof of Theorem \ref{convheat1ordsto0}]
Let $\lambda>2|\Omega|$ be fixed. Then, by the very definition of $\hat{F}^s$ in \eqref{reprs} for every $n\in\N$ we have that $D(\hat{F}^{s_n})=\mathcal{H}^{s_n}_0(\Omega)\neq\emptyset$ and, by Remark \ref{DCT} and Lemma \ref{lambdapc}, that the functionals $\hat{F}^{s_n}$ are strongly lower semicontinuous, $\lambda$-positive and $\lambda$-convex. 
Moreover, by \eqref{eq:firstvar} and by \eqref{fvf0}, for every $u\in\mathcal{H}^{s_n}_0(\Omega)$ and for every $\ffi\in\CC_\cc^\infty(\Omega)$ we have
\begin{equation*}
\lim_{t\to 0}\frac{\hat{F}^{s_n}(u+t\ffi)-\hat{F}^{s_n}(u)}{t}=\langle (-\Delta)^{s_n}u- \frac{d\omega_d}{s_n}u,\ffi\rangle_{L^2(\Omega)}\,,
\end{equation*}
which, by applying Proposition \ref{singleton} with $\F=\hat{F}^{s_n}$\,, $\h=L^2(\Omega)$ and $\hat\h=\CC_\cc^{\infty}(\Omega)$\,, implies that for every $u\in\mathcal{H}^{s_n}(\Omega)$ either $\partial\hat{F}^{s_n}(u)=\emptyset$ or $\partial\hat{F}^{s_n}(u)=\{(-\Delta)^{s_n}u-\frac{d\omega_d}{s_n}u\}$ with $(-\Delta)^{s_n}u-\frac{d\omega_d}{s_n}u\in L^2(\Omega)$\,.
Analogously, by Lemma \ref{lemma:firstvar0} and by Proposition \ref{singleton}, we have that for every $u\in\mathcal{H}^0_0(\Omega)$ either $\partial\hat{F}^{0}(u)=\emptyset$ or $\partial\hat{F}^{0}(u)=\{(-\Delta)^{0}u\}$ with $(-\Delta)^{0}u\in L^2(\Omega)$\,.

Therefore, by Theorem \ref{exist}, the solutions to the problems \eqref{cauchyord1n} ($n\in\N$) and \eqref{cauchyord1infty} are uniquely determined, and the righthand sides in \eqref{cauchyord1n} and \eqref{cauchyord1infty} belong to $L^2$ for $a.e. \, t$.
Finally, the stability claim follows by applying Theorem \ref{genstab} with $\F^n=\hat{F}^{s_n}$ and $\F^\infty=\hat{F}^0$\,, once noticed that, in view of Theorem \ref{thm:stozerofo}, assumption (b) is satisfied.
\end{proof}
 It lasts to prove Theorem~\ref{11genstab}. Also in this case, this follows from the general results already discussed. 
\begin{proof}[Proof of Theorem~\ref{11genstab}.]
	By the very definition of $F^s$ in \eqref{defFs}, we have that for all $n \in \mathbb{N}$  $D((1-s_n)F^{s_n})= \mathcal{H}_{0}^{s_n}(\Omega)\neq \emptyset$ and that the functional $(1-s_n)F^{s_n}$  is  strongly lower semicontinuous, $\lambda$-positive and $\lambda$-convex for every $\lambda>0$. Now, by  combining  Proposition~\ref{lemma:firstvar}  with  Proposition~\ref{singleton}  for  $\mathcal{F}= (1-s_n)F^{s_n}$, $\h= L^2(\Omega)$ and $\hat{\h}=\CC_{\cc}^{\infty}(\Omega)$, we have that for every $u \in \mathcal{H}_0^{s_n}(\Omega)$, either $ (1-s_n)\partial F^{s_n}(u)= \emptyset$ or $ (1-s_n)\partial F^{s_n}(u)= \{ (1-s_n)  (- \Delta)^{s_n}u \}$ with $(1-s_n) (-\Delta)^{s_n}u \in L^2(\Omega)$.
Furthermore, for every $u \in H_{0}^1(\Omega)$ and for all $\varphi \in \CC_{\cc}^{\infty}(\Omega)$ 
	we have that
	\begin{equation*}
	\lim_{h \rightarrow 0} \frac{F^1(u+h\varphi)-F^1(u)}{h}=\frac{\omega_{d}}{2}\langle \nabla u,\nabla\ffi\rangle_{L^2(\Omega)}=: \frac{\omega_{d}}{2}\langle (-\Delta)u,\ffi\rangle_{L^2(\Omega)}\,;
	\end{equation*}
	therefore, by applying Proposition \ref{singleton} with $\mathcal{F}=F^1$, $ \h= L^2(\Omega)$ and $\hat{\h}=\CC_\cc^{\infty}(\Omega)$ we have that either $ \partial F^1(u)= \emptyset$ or $\partial F^1(u)= \{ (-\Delta u)\}$ with $(-\Delta
	u) \in L^2(\Omega)$. Now, by Theorem \ref{exist}, the solutions to the problems \eqref{11caun} ($n\in\N$) and \eqref{11cauninf} are uniquely determined, and the righthand sides in \eqref{11caun} and \eqref{11cauninf} belong to $L^2$ for $a.e. \, t$. Finally, the stability claim follows by applying Theorem~\ref{genstab} with $\mathcal{F}^n=(1-s_n)F^{s_n}$ and $\mathcal{F}^{\infty}= F^1$, once noticed that, in view of  Theorem~\ref{thmGammaconvs1},  assumption (b) is satisfied.
\end{proof}

\end{document}